\documentclass{article}
\usepackage[english]{babel}
\usepackage{amsmath}
\usepackage{amsthm}
\usepackage{amssymb}
\usepackage{amscd}
\usepackage[latin1]{inputenc}
\input xy
\xyoption{all}
\newtheorem{teor}{Theorem}[section]
\newtheorem{defi}{Definition}
\newtheorem{lema}[teor]{Lemma}
\newtheorem{prop}[teor]{Proposition}
\newtheorem{cor}[teor]{Corollary}
\newtheorem{rem}[teor]{Remark}

\newtheorem{ejems}[teor]{Examples}

\usepackage{anysize}
\marginsize{2.5cm}{2.5cm}{2.5cm}{3.4cm}
\def\Hom{\mathop{\rm Hom}\nolimits}

\def\Ext {\mathop{\rm Ext}\nolimits}

\def\Ker {\mathop{\rm Ker}\nolimits}
\def\Im {\mathop{\rm Im}\nolimits}

\begin{document}
\title{Locally finitely presented categories with no flat objects}
\author{ \\
Sergio Estrada \\ sestrada@um.es \\
Departamento de Matem\'atica Aplicada \\
Universidad de Murcia \\ 30100 Murcia \\ SPAIN \\ \\
\\
\ \ Manuel Saor\'{\i}n \\ \ \ \ msaorinc@um.es \\
\ \
Departamento de Matem\'aticas \\
\ \ \ Universidad de Murcia \\
\ \ 30100 Murcia \\
SPAIN \footnote{The authors have been partially supported by the DGI
MTM2010-20940-C02-02 and by
the
Fundaci\'on Seneca 04555/GERM/06. }}
\date{}
\maketitle
\thispagestyle{empty}
\begin{abstract}
If $X$ is a quasi--compact and quasi--separated scheme, the
category
$Qcoh(X)$ of quasi--coherent sheaves on $X$ is locally finitely presented.
Therefore categorical flat quasi--coherent sheaves in
the sense of \cite{S} naturally arise. But there is also the standard
definition
of flatness in
$Qcoh(X)$ from the stalks. So it makes sense to wonder the relationship (if
any) between these two notions. In this paper we show that there are
plenty of
locally finitely presented categories having no other categorical flats than
the zero object.
As particular instance, we show that $Qcoh(\mathbf{P}^n(R)))$ has no other
categorical flat objects than zero, where $R$ is any commutative ring.
\end{abstract}



\section{Introduction}
It is widely known that locally finitely presented additive categories are
the natural framework for considering a theory of purity (see \cite{CB}). We
recall that a locally finitely presented additive category $\mathcal G$
is an
additive category with direct limits such that every object is a direct
limit of
finitely presented objects, and the class of finitely presented objects is
skeletally small. And a sequence $0\to L\stackrel{f}{\to}M
\stackrel{g}{\to}N\to 0$ in $\mathcal G$ is {\it pure} if $$0\to \Hom
(T,L)\to \Hom(T,M)\to \Hom(T,N)\to 0$$ is exact, for each finitely
presented $T$
of $\mathcal G$ (in this case $g$ is said to be a {\it pure epimorphism}).
Therefore locally finitely presented additive categories come equipped
with a
canonical notion of flat object (in the sense of Stenstr\"om \cite{S})
which is
defined in terms of pure epimorphisms (this definition is equivalent to
that of
\cite[Section 1.3]{CB}). Namely,
$F $ is {\it flat} if every epimorphism $M\to N$ is a pure epimorphism.

Moreover, if we assume that
$\mathcal G$ has enough projectives then it is not difficult to check that
flat objects are precisely the objects satisfying a Govorov-Lazard
like theorem. The Govorov-Lazard Theorem says that the closure under direct
limits
of the class of finitely generated projective objects is equal to the class
of flat objects (see \cite{G} and \cite[Th\'eor\`eme 1.2]{L}).
We point out that a flat and finitely presented object must be
necessarily projective. Let us denote by $\mathcal Flat$ the class of all
flat objects in a locally finitely presented category. Despite of
the
Govorov-Lazard theorem, this class has homological significance as well. In
\cite{CPT,Rump}
is proved, as a consequence, that $\mathcal Flat$ provides with minimal flat
resolutions that are
unique up to homotopy, so they can be used to compute right derived
functors of
the $\Hom$ functor. Also
the existence of such minimal approximations with respect to $\mathcal
Flat$ can
be used to infer the existence of pure-injective envelopes in a locally
finitely presented category (see \cite{Herzog}).

However, when working with the category $Qcoh(X)$ of quasi--coherent
sheaves on a scheme, flatness of a sheaf is defined in terms of the flatness
of its stalks. And, in general, this {\it geometrical}
notion of flatness differs from Stenstr\"om's in the non-affine case. For
instance, if $X={\bf P}^1(R)$ (with $R$ any commutative ring) the
line bundle $\mathcal O(n)$ are finitely
presented and flat, but not projective. In addition, most of the schemes $X$
that
occur in practice in algebraic geometry (like quasi--compact and
quasi--separated schemes) are such that
$Qcoh(X)$ is locally finitely presented (see \cite[I.6.9.12]{GD}), so
one also
might consider the class $\mathcal Flat$ on $Qcoh(X)$ as a natural
choice. We
point out that the
existence of unique up-to-homotopy geometrical flat resolutions is
guaranteed by \cite[Section 5]{EE}. So a natural question arises: is
there any
relation between these two notions of flatness in $Qcoh(X)$?. Whereas
geometrical flatness is well--known and widely studied in Algebraic
Geometry,
this seems not to be the case of the class $\mathcal Flat$ in
$Qcoh(X)$ ($X$ quasi--compact and quasi--separated), although in view of the
previous comments it turns out to be a canonical choice on each locally
finitely presented additive category. We note
that Crivei, Prest and
Torrecillas have recently considered this dichotomy of the two
notions of flatness in categories of sheaves of modules in \cite[Section
3]{CPT}.

Thus we devote this paper to study the class $\mathcal Flat$ of flat
objects in
locally finitely presented Gro\-then\-dieck categories. To be more
precise, we
study the {\it lack} of flat objects (apart from the zero object) in such
categories. We find sufficient conditions on a locally finitely presented
Grothendieck category to ensure that it contains no nonzero flats. They
allow
to construct a plethora of examples of locally finitely presented
Grothendieck
categories having no other flats than the zero object, showing that this
behavior is not pathological at all in locally finitely presented
categories. To
illustrate this we give examples from commutative algebra, algebraic
geometry and representation theory.

Now we will formulate a more precise statements of the resuls of the paper.
Our setup on the locally finitely presented category $\mathcal G$ is
based upon
a technical assumption on the generators of $\mathcal G$ as well as in
considering a hereditary torsion class of finite type on $\mathcal G$
(proposition \ref{prop.locally coherent without flats}). The
first set of applications (theorem \ref{teor.A-Gr
without flat objects}) deals with the quotient category of graded
modules modulo
the virtually finitely graded ones (see Section \ref{section.Quotient
categories
of
graded modules without flat
objects} for unexplained notation and
terminology):
\begin{teor}
Let $A=\oplus_{n\geq 0}A_n$ be a positively graded ring with the
property that $A_{\geq 1}$ is finitely presented and $A_{\geq n}$ is
finitely generated, as left ideals, for all $n>>0$. Let
$\mathcal{T}$ be the class of virtually finitely graded left
$A$-modules and $t$ the associated torsion radical. The category
$A-Gr/\mathcal{T}$ is locally finitely presented and, in case
$\Ext_A^1(A_0,A/t(A))=0$, it has no nonzero flat objects.
\end{teor}
The second one, theorem \ref{teor.Qcoh(X) without flat objects}, relies on
the category of quasi--coherent sheaves on a projective scheme:
\begin{teor}
Let $A=\oplus_{n\geq 0}A_n$ be a positively graded commutative ring
which is finitely generated by $A_1$ as an $A_0$-algebra, let $I$ be
the (graded) ideal of $A$ consisting of those $a\in A$ such that
$A_1^na=0$, for some $n=n(a)\in\mathbb{N}$, and let
$X=\text{Proj}(A)$ be the associated projective scheme. If the two
following conditions hold, then $Qcoh(X)$ has no nonzero
categorical flat object:
\begin{enumerate}
\item The ideal $A_{\geq 1}$ is finitely presented. \item
$\Ext_A^1(A_0,A/I)=0$
\end{enumerate}
\end{teor}
Conditions (1) and (2) of the previous theorem may look rather technical at
first glance. But it is shown in corollary \ref{cor.examples} that it covers
many natural projective schemes, like $Proj(A)$ whenever $A$ is integrally
closed or Cohen-Macaulay positively graded commutative Noetherian domain.
It also includes the case $X={\bf P}^n(R)$, the projective $n$-space
over a commutative ring $R$ (corollary \ref{cor.Qcoh(P(R)) without flat
objects}).
The last of the application (theorem
\ref{teor.modules-modulo-locally-finite})
goes back to
quotient categories of categories of representations of a quiver
with relations $(Q,\rho)$ over an arbitrary commutative ring:
\begin{teor}
Let $A$ be an algebra with enough idempotents locally finitely
presented by the quiver with relations $(Q,\rho )$, where $Q$ is
connected, locally finite and has no oriented cycles. Let
$A=B\oplus J$ its canonical Wedderburn-Malcev type decomposition and
let $\mathcal{T}$ be the class of torsion $A$-modules, whose torsion
radical is denoted by $t$. The following assertions hold:
\begin{enumerate}
\item If $\rho$ consists of homogeneous relations, $t(A)=0$ and, for
each $i\in
Q_0$, there are only finitely many vertices
$j$ such that $i\preceq j$, then $\mathcal{T}$ is a locally finitely
presented Grothendieck category with no nonzero flat objects.
\item If $Q$ is downward directed and narrow and
$\Ext_A^1(Be_i,\frac{Ae_j}{t(A)e_j})=0$, for all $i,j\in Q_0$, then
$\mathcal{G}/\mathcal{T}$ is a locally finitely presented
Grothendieck category with no nonzero flat objects.
\end{enumerate}
\end{teor}
The previous result provides an ample range of examples
of categories with no
nonzero flat objects. For instance, it applies to the category of
comodules of
a wide class of coalgebras as it is shown in corollary \ref{cor.comodules
without flat objects} (cf. \cite[Example 2.11]{CS}), and also to certain
quotient categories built up around a Cohen-Macaulay or an integrally closed
Noetherian integral domain (see corollary \ref{cor.covering of a variety}).
\section{Construction of locally finitely presented categories without flat
objects}\label{sect.construction of locally fp without flat}
All throughout these notes, the letter $\mathcal{G}$ denotes a
Grothendieck category.
\begin{defi} \label{defi.(densely) stratified}
Let $I$ be a downward directed preordered set which has not a
minimum. A set of generators $\mathcal{S}$ of $\mathcal{G}$ will be
called \emph{densely $I$-stratified} when it can be expressed as a
union $\mathcal{S}=\bigcup_{i\in I}\mathcal{S}_i$ satisfying the
following properties:
\begin{enumerate}
\item $\mathcal{S}_i\neq\emptyset$, for each $i\in I$ \item For each
$i\in I$ and each $X\in\mathcal{S}_i$, there are only finitely many
indices $j<i$ with the property that
$\Hom_\mathcal{G}(X,?)_{|\mathcal{S}_j}\neq 0$
\item For all $i, j\in I$ such that
$j<i$ and each $X\in\mathcal{S}_i$, there is an epimorphism
$X'\twoheadrightarrow X$ in $\mathcal{G}$, where $X'$ is a coproduct
of objects in $\bigcup_{k\leq j}\mathcal{S}_j$.
\end{enumerate}
We will simply say that $\mathcal{S}$ is \emph{densely stratified}
when it is densely $I$-stratified, for some downward directed
preordered set $I$ as above.
\end{defi}
\begin{defi} \label{defi.locally coherent category}
An object $X$ of $\mathcal{G}$ is called \emph{finitely presented}
when the functor
$\Hom_\mathcal{G}(X,?):\mathcal{G}\longrightarrow\text{Ab}$
preserves direct limits. The category $\mathcal{G}$ is called
\emph{locally finitely presented} when it has a set of finitely
presented generators and each object of $\mathcal{G}$ is a direct
limit of finitely presented objects.

(See \cite{S} for the next terms). If $\mathcal{G}$ is locally
finitely presented, then an epimorphism
$M\stackrel{p}{\twoheadrightarrow} N$ in $\mathcal{G}$ is called a
\emph{pure epimorphism} when the induced map
$\Hom_\mathcal{G}(X,p):\Hom_\mathcal{G}(X,M)\longrightarrow\Hom_\mathcal{G}(X,N)$
is surjective, for each finitely presented object $X$.

An object $F$ of $\mathcal{G}$ will be called \emph{flat} when every
epimorphism $M\twoheadrightarrow F$ is pure.
\end{defi}
\begin{rem}
If $\mathcal{G}$ is locally finitely presented, then the class
$fp(\mathcal{G})$ of its finitely presented objects is skeletally
small, i.e., the isoclasses of its objects form a set (see
\cite[Remark 1.9]{AR}).
\end{rem}
\begin{prop} \label{prop.no flats in densely stratified}
Let $\mathcal{G}$ be locally finitely presented and suppose that it
has a densely stratified set $\mathcal{S}$ of finitely presented
generators. Then the only flat object in $\mathcal{G}$ is the zero
object.
\end{prop}
\begin{proof}
Let $\mathcal{S}$ be densely $I$-stratified, where $I$ is a downward
directed set without minimum. Suppose that $F$ is a nonzero flat
object in $\mathcal{G}$. We then have an epimorphism
$p:\coprod_{j\in I}X_j\twoheadrightarrow F$, where $X_j$ is a
coproduct of objects in $\mathcal{S}_j$, for each $j\in I$. This
epimorphism is pure since $F$ is flat.

Let now $f:M\longrightarrow F$ be any nonzero morphism, where
$M$ is a finitely presented object in $\mathcal{G}$. There is a
finite subset $H\subset I$ together with an epimorphism
$q:\coprod_{i\in H}Y_i\twoheadrightarrow M$, where $Y_i$ is a finite
coproduct of objects of $\mathcal{S}_i$, for each $i\in H$. By
condition 2 in definition \ref{defi.(densely) stratified}, for each
$i\in H$, we know that the set
\begin{center}
$H'_i=\{j\in I:$ $j<i\text{ and
}\Hom_\mathcal{G}(Y_i,?)_{|\mathcal{S}_j}\neq 0\}$
\end{center}
is finite. Now the downward directed condition of $I$ and the
nonexistence of minimum imply that, for each $j\in I$ we can find an
element $k(j)\in I$ such that $k(j)<k$, for all $k\in(\bigcup_{i\in
H}H'_i)\cup\{j\}$. Now condition 3 of definition \ref{defi.(densely)
stratified} implies the existence of an epimorphism
$X'_j\twoheadrightarrow X_j$, where $X'_j$ is a coproduct of objects
in $\bigcup_{k\leq k(j)}\mathcal{S}_k$, for each $j\in I$. It
follows that we have an epimorphism $v:\coprod_{j\in
I}X'_j\twoheadrightarrow\coprod_{j\in I}X_j$ and, hence, another one
$p\circ v:\coprod_{j\in I}X'_j\twoheadrightarrow F$.

The epimorphism $p\circ v$ is pure since $F$ is flat and, as
consequence, there is a morphism $g:M\longrightarrow\coprod_{j\in
I}X'_j$ such that $f=p\circ v\circ g$. But then we have the morphism
$g\circ q:\coprod_{i\in H}Y_i\longrightarrow\coprod_{j\in I}X'_j$.
This morphism is necessarily zero since each $X'_j$ is a coproduct
of objects in $\bigcup_{k\leq k(j)}\mathcal{S}_k$, for all $j\in I$,
and $\Hom_\mathcal{G}(Y_i,?)_{|\mathcal{S}_k}=0$ for all such $k$.
Therefore $g\circ q=0$, which implies that $g=0$, and so $f=0$,
since $q$ is an epimorphism. We then get a contradiction with the
fact that $f\neq 0$.
\end{proof}
We now recall certain facts about localization of Grothendieck
categories (see \cite{Gabriel}, \cite{Sten}). If $\mathcal{T}$ is a
hereditary torsion class of $\mathcal{G}$ (i.e. $\mathcal{T}$ is
closed under taking subobjects, quotients, extensions and coproducts
in $\mathcal{G}$) then the inclusion functor
$i:\mathcal{T}\hookrightarrow\mathcal{G}$ has a right adjoint
$t:\mathcal{G}\longrightarrow\mathcal{T}$, and hence the unit
$1_\mathcal{T}\longrightarrow t\circ i$ of this adjunction is an
isomorphism. It follows that $(i\circ t)^2$ is naturally isomorphic
to $i\circ t$. For simplification purposes, we will abuse of
notation and denote also by $t$ the composition $i\circ t$. Then we
interpret $t$ as an idempotent functor
$t:\mathcal{G}\longrightarrow\mathcal{G}$, actually a subfunctor of
the identity, called the \emph{torsion radical} associated to
$\mathcal{T}$. For each object $M$ of $\mathcal{G}$, the subobject
$t(M)$ is the largest one which is in $\mathcal{T}$. Then $M$ fits
into an exact sequence
\begin{center}
$0\rightarrow t(M)\stackrel{incl}{\longrightarrow} M\longrightarrow
M/t(M)\rightarrow 0$,
\end{center}
which is natural on $M$. We will denote by $\mathcal{T}^\perp$ the
class of objects $Y$ such that $\Hom_\mathcal{G}(T,Y)=0$, for all
$T\in\mathcal{T}$. Then $M/t(M)$ is in $\mathcal{T}^\perp$.

There is a Grothendieck category $\mathcal{G}/\mathcal{T}$, called
the \emph{quotient category of $\mathcal{G}$ modulo $\mathcal{T}$},
together with functor
$q:\mathcal{G}\longrightarrow\mathcal{G}/\mathcal{T}$, called the
\emph{quotient functor}, which is universal with respect to each one
of the following properties:
\begin{enumerate}
\item It is an additive functor, with domain $\mathcal{G}$ and
codomain another Grothendieck category, which vanishes on all
objects of $\mathcal{T}$. \item It is an additive functor, with
domain $\mathcal{G}$ and codomain another Grothendieck category,
which takes each morphism with kernel and cokernel in $\mathcal{T}$
to an isomorphism.
\end{enumerate}
In such case, it is well-known that $q$ is exact and has a right
adjoint $\iota :\mathcal{G}/\mathcal{T}\longrightarrow\mathcal{G}$,
usually called the \emph{section functor}, which is fully faithful
(equivalently, the counit of the adjunction, $\epsilon:
q\circ\iota\longrightarrow 1_{\mathcal{G}/\mathcal{T}}$, is an
isomorphism). Moreover, if $\mu
:1_\mathcal{G}\longrightarrow\iota\circ q$ is the unit of this
adjunction, then each object $M\in\mathcal{G}$ fits into an exact
sequence
\begin{center}
$0\rightarrow
t(M)\stackrel{incl}{\longrightarrow}M\stackrel{\mu_M}{\longrightarrow}(\iota\circ
q)(M)\longrightarrow T_M\rightarrow 0$,
\end{center}
where $T_M\in\mathcal{T}$.

From the last exact sequence it easily follows that $M$ is in the
essential image of $\iota$ if, and only if,
$\Hom_\mathcal{G}(T,M)=0=\Ext_\mathcal{G}^1(T,M)$ for all
$T\in\mathcal{T}$. In the sequel, we shall freely identify
$\mathcal{G}/\mathcal{T}$ with $\Im (\iota )$ whenever necessary,
looking at $\mathcal{G}/\mathcal{T}$ in that way as a full
subcategory of $\mathcal{G}$.

The following result will be very useful later on:
\begin{prop} \label{prop.caracterization of Giraud subcategories}
Let $j:\mathcal{G}'\longrightarrow\mathcal{G}$ be an additive
functor between Grothendieck categories. Suppose that it has left
adjoint $p:\mathcal{G}\longrightarrow\mathcal{G}'$ which is exact
and such that the counit $\epsilon :p\circ j\longrightarrow
1_{\mathcal{G}'}$ is an isomorphism. Then the following assertions
hold:
\begin{enumerate}
\item $\mathcal{T}:=\text{Ker}(p)$ is a hereditary torsion class in
$\mathcal{G}$ \item The induced additive functor
$\mathcal{G}/\mathcal{T}\longrightarrow\mathcal{G}'$ is an
equivalence of categories.
\end{enumerate}
\end{prop}
\begin{proof}
1) The functor $p$ preserves coproducts since it is left adjoint. It
follows that $\mathcal{T}$ is closed under taking coproducts. On the
other hand, the exactness of $p$ implies that, given an exact
sequence $0\rightarrow L\longrightarrow M\longrightarrow
N\rightarrow 0$ in $\mathcal{G}$, one has that $M\in\mathcal{T}$ if
and only if $L$ and $N$ are in $\mathcal{T}$. Therefore
$\mathcal{T}$ is a hereditary torsion class in $\mathcal{G}$.

2) By usual properties of adjunctions, we know that $j$ is fully
faithful and,
hence, we can identify $\mathcal{G}'$ with its essential image by
$j$ and view it as a full subcategory of $\mathcal{G}$.

The universal property of the quotient category gives an induced
functor $\psi:\mathcal{G}/\mathcal{T}\longrightarrow\mathcal{G}'$
such that $\psi\circ q\cong p$, where
$q:\mathcal{G}\longrightarrow\mathcal{G}/\mathcal{T}$. Actually
$\psi$ acts on objects by taking $q(M)\rightsquigarrow p(M)$, for
each $M\in\mathcal{G}$. It immediately follows that $\psi$ is dense
(=representative), because, due to the fact that $\epsilon$ is an
isomorphism, the functor
$p$ is dense. Let us note now that if we identify $\mathcal{G}/\mathcal{T}$
with its image by the section functor $\iota
:\mathcal{G}/\mathcal{T}\longrightarrow\mathcal{G}$, then the
functor $\psi$ is just the restriction of $p$ to $Im(\iota )$.
Indeed, if $Y\in Im(\iota )$ then we have $q(Y)=Y$ and so $\psi
(Y)=p(Y)$. As a consequence, we know that the restriction of $p$ to
$Im(\iota )$ is a dense functor.

We shall prove that $Im(\iota )\subseteq\mathcal{G}'=Im(j)$, and
from this and the preceding paragraph assertion 2 will follow
since the restriction of $p$ to $Im(j)$ is the identity (recall that
$\epsilon$ is an isomorphism!).
To see
that $Im(\iota )\subseteq\mathcal{G}'$, we use the adjunction
equations. If $\eta :1_{\mathcal{G}}\longrightarrow j\circ p$ is the
unit of the adjunction $p\vdash j$, then
$p(\eta_M):p(M)\longrightarrow (p\circ j\circ p)(M)$ is an
isomorphism, for each $M\in\mathcal{G}$. Due to the exactness of $p$
this is equivalent to saying that $\Ker(\eta_M)$ and $Im(\eta_M)$ are
objects of $\mathcal{T}$. But if $M\in Im(\iota )$ then
$\Hom_\mathcal{G}(T,M)=0=\Ext_\mathcal{G}^1(T,M)$, for all
$T\in\mathcal{T}$, from which it follows $Ker(\eta_M)=0$ and, hence,
that $\eta_M$ is a section (=split monomorphism) in $\mathcal{G}$.
But $\Hom_\mathcal{G}(T,(j\circ
p)(M))\cong\Hom_{\mathcal{G}'}(p(T),p(M))=\Hom_{\mathcal{G}'}(0,p(M))=0$,
for all $T\in\mathcal{T}$. It follows that any morphism
$\text{Coker}(\eta_M)\longrightarrow (j\circ p)(M)$ is the zero
morphism, and therefore $\eta_M$ is an isomorphism. It follows that
$M\in Im(j)=\mathcal{G}'$, and hence $Im(\iota
)\subseteq\mathcal{G}'$.
\end{proof}
Recall that if $\mathcal{G}$ is locally finitely presented, then the
hereditary torsion class $\mathcal{T}$ is called of \emph{finite
type} when the section functor $\iota
:\mathcal{G}/\mathcal{T}\longrightarrow\mathcal{G}$ preserves direct
limits or, equivalently, when $\text{Im}(\iota )$ is closed under
taking direct limits in $\mathcal{G}$.

The following result seems to be folklore and it can be derived from
\cite{Rump2} (see also \cite[Proposition
A5]{Krause} for a particular case). We include here a short proof for the sake
of completeness.
\begin{prop} \label{prop.quotient of locally fp is locally fp}
Let $\mathcal{G}$ be a locally finitely presented Grothendieck
category, let $\mathcal{S}$ be a set of finitely presented
generators and let $\mathcal{T}$ be a hereditary torsion class of
finite type. Then $\mathcal{G}/\mathcal{T}$ is locally finitely
presented, the quotient functor
$q:\mathcal{G}\longrightarrow\mathcal{G}/\mathcal{T}$ preserves
finitely presented objects and $q(\mathcal{S})$ is a set of
(finitely presented) generators of $\mathcal{G}/\mathcal{T}$.
\end{prop}
\begin{proof}
Let $X$ be a finitely presented object of $\mathcal{G}$ and let
$(Y_i)_{i\in I}$ a direct system in $\mathcal{G}/\mathcal{T}$. Using
the adjunction $q\vdash\iota$ and the fact that $\mathcal{T}$ is of
finite type, we get a sequence of isomorphisms
\begin{center}
$\varinjlim\Hom_{\mathcal{G}/\mathcal{T}}(q(X),Y_i)\cong
\varinjlim\Hom_{\mathcal{G}}(X,\iota
(Y_i))\cong\Hom_{\mathcal{G}}(X,\varinjlim\iota
(Y_i))\cong\Hom_{\mathcal{G}}(X,\iota (\varinjlim
(Y_i)))\cong\Hom_{\mathcal{G}/\mathcal{T}}(q(X),\varinjlim Y_i)$,
\end{center}
whose composition coincides with the canonical morphism
$$\varinjlim\Hom_{\mathcal{G}/\mathcal{T}}(q(X),Y_i)\longrightarrow
\Hom_{\mathcal{G}/\mathcal{T}}(q(X),\varinjlim Y_i).$$ It follows
that $q(X)$ is a finitely presented object of
$\mathcal{G}/\mathcal{T}$. Therefore $q$ preserves finitely
presented objects.

On the other hand, each object of $\mathcal{G}/\mathcal{T}$ is of
the form $q(M)$, with $M\in\mathcal{G}$. If we put $M=\varinjlim
X_i$, with the $X_i$ finitely presented, then we get
$q(M)=\varinjlim q(X_i)$ since the functor $q$ preserves direct
limits. It follows that each object of $\mathcal{G}/\mathcal{T}$ is
a direct limit of finitely presented objects.

Finally, for $M$ as above, we have an epimorphism $\coprod_{j\in
J}S_j\twoheadrightarrow M$ in $\mathcal{G}$, where the $S_j$ are in
$\mathcal{S}$. Then we get an epimorphism $\coprod_{j\in
J}q(S_j)\twoheadrightarrow q(M)$ in $\mathcal{G}/\mathcal{T}$. This
shows that $q(\mathcal{S})$ is a set of finitely presented
generators of $\mathcal{G}/\mathcal{T}$ and, hence, that
$\mathcal{G}/\mathcal{T}$ is locally finitely presented.
\end{proof}
Recall that if $\mathcal{X}$ is any set of objects in $\mathcal{G}$
and $M$ is another object, then the \emph{trace} of $\mathcal{X}$ on
$M$ is the sum of all the images of morphisms $f:X\longrightarrow
M$, with $X\in\mathcal{X}$. Given a hereditary torsion class
$\mathcal{T}$ in $\mathcal{G}$, with torsion radical $t$, and a
class $\mathcal{X}$ of objects in $\mathcal{G}$, we shall denote by
$\bar{\mathcal{X}}$ the class of objects $X/t(X)$, with
$X\in\mathcal{X}$. Also, given an object $X\in\mathcal{G}$, we
denote by $\mathcal{T}_X$ the class of objects $T\in\mathcal{T}$
which are epimorphic image of $X$.
\begin{prop} \label{prop.locally coherent without flats}
Let $I$ be a downward directed preordered set without minimum and
let $\mathcal{G}$ be a locally finitely presented Grothendieck
category admitting a set $\mathcal{S}$ of finitely presented
generators decomposed as a union $\mathcal{S}=\bigcup_{i\in
I}\mathcal{S}_i$ of nonempty subsets. Let us denote by $tr_j(M)$ the
trace of $\bigcup_{k\leq j}\mathcal{S}_k$ in $M$, for each
$M\in\mathcal{G}$ and each $j\in I$. Suppose that $\mathcal{T}$ is a
hereditary torsion class of finite type, with torsion radical $t$,
satisfying the following conditions:
\begin{enumerate}
\item For each $i\in I$ and each $X\in\mathcal{S}_i$, the set of
indices $j<i$ such that either
$\Hom_A(X,?)_{|\bar{\mathcal{S}}_j}\neq 0$ or
$\Ext_\mathcal{G}^1(?,?)_{|\mathcal{T}_X\times\bar{\mathcal{S}}_j}\neq
0$
is finite.
\item For each pair of indices $j<i$ in $I$ and each
$X\in\mathcal{S}_i$, one has that $X/tr_j(X)\in\mathcal{T}$.
\end{enumerate}
Then $\mathcal{G}/\mathcal{T}$ is locally finitely presented and
contains no nonzero flat object.
\end{prop}
\begin{proof}
By proposition \ref{prop.quotient of locally fp is locally fp},
$\mathcal{G}/\mathcal{T}$ is locally finitely presented and
$q(\mathcal{S})$ is a set of finitely presented generators of
$\mathcal{G}/\mathcal{T}$.

We shall prove that $q(\mathcal{S})=\bigcup_{i\in
I}q(\mathcal{S}_i)$ is a densely $I$-stratified set of generators
of $\mathcal{G}/\mathcal{T}$ and the result will follow from
proposition \ref{prop.no flats in densely stratified}. By condition
$1$, given an index $i\in I$ and an object $X\in\mathcal{S}_i$, the
set $I_X$ of indices $j<i$ such that
$\Hom_\mathcal{G}(X,?)_{|\bar{\mathcal{S}}_j}=
0=\Ext_\mathcal{G}^1(?,?)_{|\mathcal{T}_X\times\bar{\mathcal{S}}_j}$
is cofinite within the set of indices $j\in I$ such that $j<i$. We
claim that if $j\in I_X$, then
$\Hom_{\mathcal{G}/\mathcal{T}}(q(X),q(Y))=0$, for all
$Y\in\mathcal{S}_j$. To see that, take $Y\in\mathcal{S}_j$ and apply
the functor $\Hom_\mathcal{G}(X,?)$ to the exact sequence
\begin{center}
$0\rightarrow Y/t(Y)\stackrel{\mu}{\longrightarrow} (\iota\circ
q)(Y)\stackrel{p}{\longrightarrow}T\rightarrow 0$
\end{center}
where $\mu$ comes from the unit map of the adjunction $q\vdash\iota
$. We know that $T\in\mathcal{T}$. If now $f:X\longrightarrow T$ is
any morphism, then $Im(f)\in\mathcal{T}_X$ and hence
$\Ext_\mathcal{G}^1(Im(f),Y/t(Y))=0$. But then the inclusion
$\text{Im}(f)\hookrightarrow T$ factors through $p$. This implies
that $\text{Im}(f)=0$ since
$\text{Hom}_\mathcal{G}(\text{Im}(f),(\iota\circ q)(Y))=0$.
Therefore we have $\Hom_\mathcal{G}(X,T)=0$ and, hence, we also have
$0=\Hom_\mathcal{G}(X,(\iota\circ
q)(Y))\cong\Hom_{\mathcal{G}/\mathcal{T}}(q(X),q(Y))$ because
$\Hom_\mathcal{G}(X,Y/t(Y))=0$ by hypothesis.

The last paragraph shows that, given $i\in I$ and
$X\in\mathcal{S}_i$, the set of indices $j<i$ such that
$\Hom_{\mathcal{G}/\mathcal{T}}(q(X),?)_{|q(\mathcal{S}_j)}\neq 0$
is finite. Therefore $q(\mathcal{S})$ satisfies condition $2$ of
definition \ref{defi.(densely) stratified}.

We now prove that $q(\mathcal{S})$ satisfies condition 3 of in that
definition. Indeed if $j<i$ then, for each $X\in\mathcal{S}_i$, we
consider the canonical morphism in $\mathcal{G}$
\begin{center}
$f:\coprod_{S\in\bigcup_{k\leq
j}\mathcal{S}_k}S^{(\Hom_\mathcal{G}(S,X))}\longrightarrow X$,
\end{center}
whose image is precisely $tr_j(X)$. By condition 2, we know that
$Coker(f)\in\mathcal{T}$. But then the induced morphism in
$\mathcal{G}/\mathcal{T}$
\begin{center}
$\coprod_{S\in\bigcup_{k\leq
j}\mathcal{S}_k}q(S)^{(\Hom_\mathcal{G}(S,X))}\cong
q(\coprod_{S\in\bigcup_{k\leq
j}\mathcal{S}_k}S^{(\Hom_\mathcal{G}(S,X))})\stackrel{q(f)}{\longrightarrow}
q(X)$,
\end{center}
is an epimorphism, because $q$ is exact and vanishes on
$\mathcal{T}$.
\end{proof}
\section{Quotient categories of graded modules without flat
objects}\label{section.Quotient categories of graded modules without flat
objects}
All throughout this section $A=\oplus_{n\geq 0}A_n$ is a positively
graded ring. Further hypotheses will be imposed on it when stating
the main results. We denote by $A-Gr$ the category of
($\mathbb{Z}$-)graded $A$-modules. It is a Grothendieck category. For
each $n\in\mathbb{Z}$, there is a canonical automophism
$?[n]:A-Gr\stackrel{\cong}{\longrightarrow}A-Gr$, called the
\emph{$n$-th shift}. For each $M\in A-Gr$, $M[n]$ is the graded
$A$-module with the same underlying (ungraded) $A$-module as $M$,
but with grading given by $M[n]_i=M_{n+i}$. It is well-known (cf.
\cite{NasFred}) that $A-Gr$ is locally finitely presented, with
$\{A[n]:$ $n\in\mathbb{Z}\}$ as a set of finitely generated
projective generators.

We will use some standard facts about $\Hom$ and $\Ext$ for graded modules
(see, e.g., \cite[Chapter I]{NasFred}). If $M$ and $N$ are graded
$A$-modules and we put
$\text{HOM}_A(M,N)=\oplus_{n\in\mathbb{Z}}\Hom_{A-Gr}(M,N[n])$, then
$\text{HOM}_A(M,N)$ is a graded abelian group with
$\text{HOM}_A(M,N)_n=\Hom_{A-Gr}(M,N[n])$, for each
$n\in\mathbb{Z}$. Then one gets a covariant functor
$\text{HOM}_A(M,?):A-Gr\longrightarrow\mathbb{Z}-Gr$. Its right
derived functors are denoted by $\text{EXT}_R^i(M,?)$ ($i\geq 0$).
It is immediate to see that
$\text{EXT}_A^i(M,N)=\oplus_{n\in\mathbb{Z}}\Ext_{A-Gr}^i(M,N[n])\cong\oplus_{n\in\mathbb{Z}}\Ext_{A-Gr}^i(M[-n],N)$,
which is a graded abelian group with
$\text{EXT}_A^i(M,N)_n=\Ext_{A-Gr}^i(M,N[n])\cong\Ext_{A-Gr}^i(M[-n],N)$,
for each $n\in\mathbb{Z}$.

On the other hand, when one forgets the grading, $\text{HOM}_A(M,N)$
is a subgroup of $\Hom_A(M,N)$ which coincides with it in case $M$
is finitely generated (see \cite[Corollary I.2.11]{NasFred}). If, in
addition, $M$ has a projective resolution
\begin{center}
$\ldots\to P^{-n}\rightarrow \ldots\rightarrow P^{-1}\rightarrow
P^0\rightarrow M\rightarrow 0$
\end{center}
such that $P^{-i}$ is finitely generated for $i=0,1,\ldots,r$, then
the two compositions of functors
\begin{center}
$A-Gr\xrightarrow{\text{forgetful}}A-Mod\xrightarrow{\Ext_A^i(M,?)}\mathbb{Z}
-\text{Mod}$
$A-Gr\xrightarrow{\text{EXT}_A^i(M,?)}\mathbb{Z}-Gr\xrightarrow{
\text { forgetful}}\mathbb{Z}-\text{Mod}$
\end{center}
are naturally isomorphic, for all $0\leq i<r$.

If $A=\oplus_{n\geq 0}A_n$ is any positively graded ring, then we
have a homomorphism of graded rings $\pi:A\longrightarrow A_0$
mapping each $a\in A$ onto its $0$-homogeneous component $a_0$ (here
we are viewing $A_0$ as a graded ring concentrated in degree $0$).
Then each $A_0$-module, and in particular $A_0$ itself, becomes
canonically a graded $A$-module concentrated in degree $0$.

Recall that a $T\in A-Gr$ is called \emph{finitely graded} when
$T_n=0$, for all but finitely many integers $n$. We shall say that
$T$ is \emph{virtually finitely graded} when each finitely generated
(or cyclic) graded submodule of $T$ is finitely graded. All
throughout this section $\mathcal{T}$ will denote the class of
virtually finitely graded modules. Recall (cf. \cite{NasFred}) that a
hereditary torsion theory in $A-Gr$ is called \emph{rigid} when
$T[n]\in\mathcal{T}$, for all $T\in\mathcal{T}$ and
$n\in\mathbb{Z}$.
\begin{lema} \label{lem.description of Im(i)}
Suppose that there is an integer $k>0$ such that $A_{\geq
m}:=\oplus_{n\geq m}A_n$ is finitely generated as left ideal, for
all $m\geq k$. The class $\mathcal{T}$ is a rigid hereditary
torsion class in $A-Gr$. Moreover, if $A_{\geq 1}:=\oplus_{n\geq
m}A_1$ is finitely presented as a left ideal and $\iota
:A-Gr/\mathcal{T}\longrightarrow A-Gr$ is the section functor, then:
\begin{enumerate}
\item $\Im (\iota )$ consists of the graded
$A$-modules $Y$ such that $\Hom_A(A_0,Y)=0=\Ext_A^1(A_0,Y)$
\item $\mathcal{T}$ is of finite type.
\end{enumerate}
\end{lema}
\begin{proof}
The class $\mathcal{T}$ is clearly closed under taking subobjects,
quotients and arbitrary coproducts, even without the existence of
the integer $k$ of the statement. Note that an alternative
description of a $T\in\mathcal{T}$ is that it is a graded module
such that each of its homogeneous elements is annihilated by
$A_{\geq m}$, for some $m\geq 0$. Suppose now that $0\rightarrow
T'\longrightarrow M\longrightarrow T\rightarrow 0$ is an exact
sequence in $A-Gr$, with $T,T'\in\mathcal{T}$. For each homogeneous
element $x\in M$, the cyclic graded module $Ax+T'/T'$ is isomorphic
to a graded submodule of $T$. Then there exists an integer $m>0$
such that $A_{\geq m}x\subset T'$. If we choose $m\geq k$, where $k$
is the integer in the statement, then $A_{\geq m}x$ is a finitely
generated graded submodule of $T'$. We get that $A_{\geq m}x$ is
finitely graded and $Ax/A_{\geq m}x$ is obviously finitely graded.
It follows that $Ax$ is finitely graded, and hence $\mathcal{T}$ is
a hereditary torsion class in $A-Gr$, which is clearly rigid.

Suppose now that $A_{\geq 1}$ is finitely presented as a left ideal
and let us prove assertions 1 and 2.
If $Y\in\Im(\iota )$ then
$\Hom_{A-Gr}(A_0[n],Y)=0=\Ext_{A-Gr}^1(A_0[n],Y)$ since $A_0$ is in
$\mathcal{T}$ when we view it as a graded module concentrated in
degree zero. But the comments preceding this lemma show that we have
equalities $\Hom_A(A_0,Y)=\text{HOM}_A(A_0,Y)$ and
$\Ext_A^1(A_0,Y)=\text{EXT}_A^1(A_0,Y)$ and, hence, we get that
$\Hom_A(A_0,Y)=0=\Ext_A^1(A_0,Y)$.

Conversely, suppose that $\Hom_A(A_0,Y)=0=\Ext_A^1(A_0,Y)$. Then
$\Hom_A(T,Y)=0$, for all $T\in\mathcal{T}$ (i.e.
$Y\in\mathcal{T}^\perp$). Indeed, if $f:T\longrightarrow Y$ is a
nonzero morphism in $A-Gr$, then $Im(f)\in\mathcal{T}$. Then each
finitely generated submodule of $Im(f)$ is finitely graded and,
hence, it contains a nonzero homogeneous element annihilated by
$A_{\geq 1}$. We then obtain a nonzero morphism
$A_0[n]\cong\frac{A}{A_{\geq 1}}[n]\longrightarrow
Im(f)\hookrightarrow Y$, which contradicts the fact that
$\Hom_A(A_0,Y)=0$.

By last paragraph, the task reduces to prove that if
$Y\in\mathcal{T}^\perp$ is such that $\Ext_A^1(A_0,Y)=0$, then
$\Ext_{A-Gr}^1(T,Y)=0$ holds for all $T\in\mathcal{T}$. Let
$T\in\mathcal{T}$ be any object, which we express as a direct union
$\bigcup_{i\in I}T_i$ of its finitely generated graded submodules.
Then we have an exact sequence in $A-Gr$
\begin{center}
$0\rightarrow T' \longrightarrow\coprod_{i\in
I}T_i\stackrel{\pi}{\longrightarrow} T\rightarrow 0$,
\end{center}
where $\pi$ is the canonical projection onto the direct limit. Then
we get an exact sequence of abelian groups
\begin{center}
$0=\Hom_{A-Gr}(T',Y)\longrightarrow\Ext_{A-Gr}^1(T,Y)\longrightarrow\Ext_{A-Gr}^1(\coprod_{i\in
I}T_i,Y)$.
\end{center}
Bearing in mind that $\Ext_{A-Gr}^1(?,Y):A-Gr\longrightarrow Ab$
takes coproducts to products, our task is reduced to check that
$\Ext_{A-Gr}^1(T,Y)=0$, whenever $T\in\mathcal{T}$ is finitely
generated. By definition of $\mathcal{T}$, if $T\in\mathcal{T}$ is
finitely generated, then it is concentrated in degrees
$\{r,r+1,...,s\}$, for some integers $r<s$. Then the chain
\begin{center}
$0\subset T_s\subset (T_{s-1}\oplus T_s)\subset ...\subset
(T_{r+1}\oplus...\oplus T_s)\subset \oplus_{r\leq k\leq s}T_k=T$
\end{center}
is a finite ascending chain in $A-Gr$ on which each quotient of a
member by the previous one is annihilated by $A_{\geq 1}$ and is
concentrated in just one degree. The proof then reduces to the case
when $T=X[n]$, for some $A_0$-module $X$ (viewed as a graded
$A$-module concentrated in degree $0$) and some $n\in\mathbb{Z}$.
Fix such $X$ and $n$ and consider an exact sequence of $A_0$-modules
\begin{center}
$0\rightarrow Z\longrightarrow A_0^{(H)}\longrightarrow X\rightarrow
0$,
\end{center}
which we view as a sequence of graded $A$-modules which are
annihilated by $A_{\geq 1}$. We then get an exact sequence of
abelian groups
\begin{center}
$0=\Hom_{A-Gr}(Z[n],Y)\longrightarrow\Ext_{A-Gr}^1(X[n],Y)\longrightarrow\Ext_{A-Gr}^1(A_0^{(H)}[n],Y)$.
\end{center}
The proof of assertion 1 will be finished if we prove that
$\Ext_{A-Gr}^1(A_0[n],Y)=0$, for all $n\in\mathbb{Z}$, since
$\Ext_{A-Gr}^1(?,Y)$ takes coproducts to products. But
$\Ext_{A-Gr}^1(A_0[n],Y)$ is just the $-n$-th homogeneous component
of $\text{EXT}_A^1(A_0,Y)=\Ext_A^1(A_0,Y)$, which is zero by
hypothesis.

In order to prove assertion 2, note that $\mathcal{T}$ is of finite
type if, and only if, the class $\Im(\iota )$ is closed under taking
direct limits in $A-Gr$. Bearing in mind that $A_0$ admits a
projective resolution $$\ldots\rightarrow P^{-2}\rightarrow
P^{-1}\rightarrow P^0\rightarrow A_0\rightarrow 0$$ in $A-Gr$, where
$P^{-i}$ is finitely generated for $i=0,1,2$, we get that the
composition
\begin{center}
$A-Gr\xrightarrow{\text{forgetful}}A-Mod\xrightarrow{\Ext_A^1(A_0,
?) }\mathbb{Z}-\text{Mod}$
\end{center}
preserves direct limits. This together with assertion 1 show that
$\Im(\iota )$ is closed under direct limits.
\end{proof}
We are now ready to prove the main result of this section.
\begin{teor} \label{teor.A-Gr without flat objects}
Let $A=\oplus_{n\geq 0}A_n$ be a positively graded ring with the
property that $A_{\geq 1}$ is finitely presented and $A_{\geq n}$ is
finitely generated, as left ideals, for all $n>>0$. Let
$\mathcal{T}$ be the class of virtually finitely graded left
$A$-modules and $t$ the associated torsion radical. The category
$A-Gr/\mathcal{T}$ is locally finitely presented and, in case
$\Ext_A^1(A_0,A/t(A))=0$, it has no nonzero flat objects.
\end{teor}
\begin{proof}
By lemma \ref{lem.description of Im(i)}, we know that $\mathcal{T}$
is a hereditary torsion class of finite type, and then, by
proposition \ref{prop.quotient of locally fp is locally fp}, we know
that $A-Gr/\mathcal{T}$ is locally finitely presented.

Let us put $I=\mathbb{Z}$ and $\mathcal{S}_n=\{A[n]\}$, for each
$n\in\mathbb{Z}$. We shall prove that the two conditions of
proposition \ref{prop.locally coherent without flats} hold for
$\mathcal{S}=\bigcup_{n\in\mathbb{Z}}\mathcal{S}_n$. Note that
$\bar{\mathcal{S}}_n=\{\frac{A}{t(A)}[n]\}$, for each
$n\in\mathbb{Z}$. We have that
$\Hom_{A-Gr}(A[m],\frac{A}{t(A)}[n])\cong (\frac{A}{t(A)})_{n-m}$,
which is zero when $n<m$. Then the set of indices mentioned in
condition 1 of proposition \ref{prop.locally coherent without flats}
is empty in this case.

On the other hand, if $m<n$ are integers then
we have $tr_m(A[n])=tr_{m-n}(A)[n]$ and hence
$\frac{A[n]}{tr_m(A[n])}\cong\frac{A}{tr_{m-n}(A)}[n]$ and, in order
to check condition 2 of proposition \ref{prop.locally coherent
without flats}, it is enough to check that
$A/tr_{-r}(A)\in\mathcal{T}$, for all $r>0$. But the trace of
$\bigcup_{k\leq -r}\mathcal{S}_k=\{A[k]:\text{ }k\leq -r\}$ in $A$
is precisely $A_{\geq r}=\oplus_{n\geq r}A_n$, and hence
$A/tr_{-r}(A)=A/A_{\geq r}$. This is a finitely graded $A$-module
and, hence, it is in $\mathcal{T}$.
\end{proof}
\section{The category of quasicoherent sheaves on a projective scheme has no
nonzero flat objects}
\label{section.The category of quasicoherent sheaves on a projective
scheme has
no
nonzero flat objects}
Recall that if $A=\oplus_{n\geq 0}A_n$ is a positively graded
commutative ring,
then $\text{Proj}(A)$ is the set of graded prime ideals $\mathbf{p}$
of $A$ such that $A_{\geq 1}\not\subset\mathbf{p}$. It is a scheme
(cf. \cite[Chapter II]{H}) whose underlying topology is the Zariski
one. The (Grothendieck) category of quasicoherent sheaves on
$\text{Proj}(A)$ is our next objective. When $M$ is a graded
$A$-module and $\mathbf{p}\in\text{Proj}(A)$, we shall denote by
$M_{(\mathbf{p})}$ the graded localization of $M$ at $\mathbf{p}$.
It consists of the fractions $\frac{x}{s}$, where $x\in M$ and $s\in
A\setminus\mathbf{p}$ are homogeneous elements with the same degree.
\begin{lema} \label{lem.the class T in Serre's situation}
Let $A=\oplus_{n\geq 0}A_n$ be a commutative positively graded ring
which is finitely generated by $A_1$ as an $A_0$-algebra. The
class $\mathcal{T}$ of virtually finitely graded
$A$-modules is a rigid hereditary torsion class in $A-Gr$ and it
consits of those $M\in A-Gr$ such that $M_{(\mathbf{p})}=0$, for all
$\mathbf{p}\in\text{Proj} (A)$.
\end{lema}
\begin{proof}
If $\{x_0,...,x_n\}$ is a set of elements of $A_1$ which generates
$A$ as an $A_0$-algebra, then the graded ideal $A_{\geq m}$ is
generated by all monomials of degree $m$ in the $x_i$. It follows
that $A_{\geq m}$ is a finitely generated ideal, for each $m>0$. By
lemma \ref{lem.description of Im(i)}, we get that $\mathcal{T}$ is a
rigid hereditary torsion class in $A-Gr$.

To prove the final statement, due to the exactness of localization,
we just need to check that a cyclic graded $A$-module $M$ is
finitely graded if, and only if, $M_{(\mathbf{p})}=0$ for all
$\mathbf{p}\in\text{Proj}(A)$. For the 'only if' part note that if
$M$ is finitely graded, then there is a power $(A_{\geq 1})^m$ which
annihilates $M$. Choosing an element $s_\mathbf{p}\in A_{\geq
1}\setminus\mathbf{p}$, we get that $s_\mathbf{p}^mM=0$ and
therefore $M_{(\mathbf{p})}=0$, for each
$\mathbf{p}\in\text{Proj}(A)$.

For the 'if' part, suppose $M=Ax$, where $x$ is a homogenous element
of degree $r$. Then, applying the shift $?[-r]:A-Gr\longrightarrow
A-Gr$, we can and shall assume, without loss of generality, that
$deg(x)=0$. Then $M$ is isomorphic to $A/I$, for some graded ideal
$I$ of $A$. Saying that $M_{(\mathbf{p})}=0$, for all $\mathbf{p}\in
\text{Proj}(A)$ is equivalent to saying that the graded ring $A/I$
has $\text{Proj}(A/I)=\emptyset$. This means that $(A/I)_{\geq
1}=\frac{A_{\geq 1}+I}{I}$ is contained in all the graded prime
ideals of $A/I$. That is, the homogeneous elements of $(A/I)_{\geq
1}$ are all nilpotent, and then $(A/I)_{\geq 1}$ is a nilpotent
ideal of $A/I$ since $(A/I)_{\geq 1}$ is a finitely generated ideal
of $A/I$. This means that there is a large enough integer $m>0$ such
that $(A_{\geq 1})^n\subseteq I$, for all $n\geq m$. But, due to the
fact that $A_1$ generates $A$ as an $A_0$-algebra, we have an
equality $A_n=A_1^n$, for each $n>0$. It follows that $A_n\subset
I$, and hence $(A/I)_n=0$, for all $n\geq m$. Then $M\cong A/I$ is
finitely graded.
\end{proof}
For any (algebraic) scheme $X$, we denote by $Qcoh(X)$ the
category of quasi-coherent sheaves on $X$. The following result,
essentially due to Serre, is standard. For the convenience of the
reader, we give a summary of the proof essentially derived from
Murfet's notes (see \cite{Mur} and \cite{Mur2}).
\begin{teor}[Serre] \label{teor.Serre}
Let $A=\oplus_{n\geq 0}A_n$ be a positively graded commutative
algebra which is finitely generated by $A_1$ as an $A_0$-algebra,
let $X=\text{Proj}(A)$ be the associated projective scheme and let
$\mathcal{T}$ be the class of virtually finitely graded $A$-modules.
There is an equivalence of categories
$A-Gr/\mathcal{T}\stackrel{\cong}{\longleftrightarrow}Qcoh(X)$.
\end{teor}
\begin{proof}
Using the notation of \cite[Chapter II, Section 5]{H}, the
assignment $M\rightsquigarrow\tilde{M}$ gives a functor
$(-)^\thicksim :A-Gr\longrightarrow O_X-\text{Mod}$ which is exact
(see, e.g., \cite[Lemma 2]{Mur2}) and, by definition of
quasicoherent sheaf, this functor has $Qcoh(X)$ as its
essential image. Then the induced functor
$p:A-Gr\longrightarrow Qcoh(X)$ is also exact because
$Qcoh(X)$ is a full subcategory of $O_X-\text{Mod}$.

On the other hand, there is an
additive functor $\Gamma_*:O_X-Mod\longrightarrow A-Gr$ which is
right adjoint to $(-)^\thicksim $ (see \cite[Proposition 2]{Mur}).
We shall denote by $\epsilon
:(-)^\thicksim\circ\Gamma_*\longrightarrow 1_{O_X-MOd}$ and $\eta
:1_{A-Gr}\longrightarrow\Gamma_*\circ (-)^\thicksim$ the counit and
the unit of this adjunction, respectively. It follows that $p$ is
left adjoint to the restriction
$j:=\Gamma_{*|Qcoh(X)}:Qcoh(X)\longrightarrow A-Gr$. By
\cite[Proposition II.5.15]{H}, we know that the counit of this
adjunction is the identity.

Now proposition \ref{prop.caracterization of Giraud subcategories}
tells us that $Qcoh(X)$ is equivalent to $A-Gr/\Ker(p)$ and our
task reduces to check that $\Ker (p)=\mathcal{T}$. But, by the
properties of sheaves, we know that $\tilde{M}=0$ if, and only if,
each stalk $\tilde{M}_\mathbf{p}$ is zero. By \cite[Proposition
II.5.11]{H}, this is equivalent to
saying that $M_{(\mathbf{p})}=0$,
for all $\mathbf{p}\in\text{Proj}(A)$. The preceding lemma states
that this is equivalent to saying that $M\in\mathcal{T}$.
\end{proof}
\begin{rem}
In the theory of (algebraic schemes) there is a classical concept of
flat sheaf, which we call here \emph{geometric flat}. Namely, if $X$
is a scheme with structural sheaf of rings $O_X$, then a sheaf
$\mathcal{F}$ of $O_X$-modules is geometrically flat when the stalk
$\mathcal{F}_x$ is a flat $O_{X,x}$-module, for each point $x\in X$.
For this reason, a flat object in a Grothendieck category
$\mathcal{G}$ will be called \emph{categorical flat} in the
particular case when $\mathcal{G}=Qcoh(X)$, for a scheme $X$.
\end{rem}
The main result of this section is the following.
\begin{teor} \label{teor.Qcoh(X) without flat objects}
Let $A=\oplus_{n\geq 0}A_n$ be a positively graded commutative ring
which is finitely generated by $A_1$ as an $A_0$-algebra, let $I$ be
the (graded) ideal of $A$ consisting of those $a\in A$ such that
$A_1^na=0$, for some $n=n(a)\in\mathbb{N}$, and let
$X=\text{Proj}(A)$ be the associated projective scheme. If the two
following conditions hold, then $Qcoh(X)$ has no nonzero
categorical flat object:
\begin{enumerate}
\item The ideal $A_{\geq 1}$ is finitely presented. \item
$\Ext_A^1(A_0,A/I)=0$
\end{enumerate}
\end{teor}
\begin{proof}
Note that a graded $A$-module is finitely graded if, and only if, it
is annihilated by $A_n=A_1^n$, for some $n\geq 0$. This shows that
$I=t(A)$, where $t$ is the torsion radical associated to
$\mathcal{T}$.

By the proof of lemma \ref{lem.the class T in
Serre's situation}, we know that $A_{\geq m}$ is finitely generated,
for all $m\geq 0$. Then, by theorem \ref{teor.A-Gr without flat
objects}, we know that $A-Gr/\mathcal{T}$ does not have nonzero flat
objects. Finally, by Serre's theorem \ref{teor.Serre}, we get the
result.
\end{proof}
We refer to \cite{mat} or \cite{Kunz} for the terminology used in
the following result.
\begin{cor} \label{cor.examples}
Let $A=\oplus_{n\geq 0}A_n$ be a positively graded commutative
Noetherian domain which is finitely generated by $A_1$ as an
$A_0$-algebra, suppose that $A_{\geq 1}$ has height $\geq 2$ and let
$X=\text{Proj}(A)$ be the associated projective scheme. If $A$
satisfies one of the following two conditions (as an ungraded ring),
then $Qcoh(X)$ has no nonzero categorical flat objects.
\begin{enumerate}
\item $A$ is integrally closed \item $A$ is Cohen-Macaulay
\end{enumerate}
\end{cor}
\begin{proof}
In both cases the ideal $I$ of theorem \ref{teor.Qcoh(X) without
flat objects} is zero.

1) If $A$ is integrally closed, then $A$ satisfies Serre condition
$S_2$ (see \cite[Chapter VII, Section 6]{Sten}) for the definition
and the result). Now proposition VII.6.8 in \cite{Sten} says that
$A/A_{\geq 1}=A_0$ satisfies the property that
$\text{Hom}_A(A_0,E^0(A)\oplus E^1(A))=0$, where
$0\rightarrow
A\longrightarrow E^0(A)\longrightarrow E^1(A)$ is the initial part
of the minimal injective resolution of $A$ in $A-\text{Mod}$. It
follows immediately that $\Ext_A^1(A_0,A)=0$ and, hence, theorem
\ref{teor.Qcoh(X) without flat objects} applies.

2) If $A$ is Cohen-Macaulay, then $\text{depth}(A_{\geq 1},A)$
coincides with the height $\text{ht}(A_{\geq 1})$ (see
\cite[Theorem VI.3.14]{Kunz}). By definition of the depth, this
means that $0=\Ext_A^i(A/A_{\geq 1},A)=\Ext_A^i(A_0,A)$, for all
$i<\text{ht}(A_{\geq 1})$. In particular $\Ext_A^1(A_0,A)=0$ and
theorem \ref{teor.Qcoh(X) without flat objects} applies again.
\end{proof}
Recall that if $R$ is a commutative ring and $n>0$ is an integer,
then the projective $n$-space over $R$ is
$\mathbf{P}^n(R)=\text{Proj}(R[x_0,...,x_n])$, where the $x_i$ are
variables over $R$ of degree $1$.
\begin{cor} \label{cor.Qcoh(P(R)) without flat objects}
$Qcoh(\mathbf{P}^n(R))$ does not have nonzero categorical
flat objects, for any commutative ring $R$ and any integer $n>0$.
\end{cor}
\begin{proof}
Put $A:=R[x_0,x_1,...,x_n]$ with its canonical grading. Here
$A_{\geq 1}=\sum_{0\leq i\leq n}Ax_i$ is the ideal generated by
$\{x_0,x_1,...,x_n\}$ and the ideal $I$ of theorem \ref{teor.Qcoh(X)
without flat objects} is zero. Moreover, we have $A_0=R$, viewed as
$A$-module in the usual way and $A_1=\sum_{0\leq i\leq n}Rx_i$.
Therefore $A$ is generated by $A_1$ as an $A_0$-algebra. The result
will follow from theorem \ref{teor.Qcoh(X) without flat objects}
once we check that $\Ext_A^1(R,A)=0$. But this is well-known. It is
a straightforward consequence of \cite[Chapter VIII, exercise 7]{CE}
or, by finite induction, from \cite[Theorem 9.37]{Rot}.
\end{proof}
\section{Module categories modulo locally finite dimensional
ones}\label{section.Module categories modulo locally finite dimensional
ones}
In this section we fix a commutative ring $R$ with $1$ and the term
'algebra' will mean (not necessarily unital, but always associative)
'$R$-algebra'. An \emph{algebra with enough idempotents} is an
associative algebra $A$ on which a distinguished family of nonzero
orthogonal idempotents $(e_i)_{i\in I}$ has been fixed, satisfying
that $\oplus_{i\in I}e_iA=A=\oplus_{i\in I}Ae_i$ as an $R$-module.
We assume all the time that $Re_i=e_iR$ is free of rank $1$ over $R$
generated by $e_i$ and put $B=\oplus_{i\in I}Re_i$, which is a
commutative subalgebra of $A$. We will assume that $A$ is
\emph{augmented} (with respect to $B$). This means that the
inclusion $B\hookrightarrow A$ is a section in the category of
algebras or, equivalently, that there exists a two-sided ideal $J$
of $A$ such that $A=B\oplus J$ as $B$-bimodules. Although $J$ is not
unique, we will fix one and the corresponding decomposition
$A=B\oplus J$ will be called the \emph{Wedderburn-Malcev type
decomposition} of $A$.
Over such an algebra $A$, a (left) module $M$ will be always
considered to be unitary, i.e. $AM=M$ or, equivalently, it admits a
decomposition $M=\oplus_{i\in I}e_iM$ as an $R$-module. The
corresponding category is denoted by $A-\text{Mod}$ and is a
Grothendieck category with $\{Ae_i:$ $i\in I\}$ as a set of finitely
generated projective generators.
We call
an $A$-modulo $M$ \emph{torsion} (with respect to the given
Wedderburn-Malcev decomposition) when each element of $M$ is
annihilated by some power $J^n$. In this section, unless explicitly
said otherwise, we denote by $\mathcal{T}$ the full subcategory of
$A-\text{Mod}$ consisting of torsion $A$-modules.

The following is the analogous of lemma \ref{lem.description of
Im(i)} in this context.
\begin{lema} \label{lem.Mod-modulo-lfd has no flats}
Let $A$ an algebra with enough idempotents as above, with
$(e_i)_{i\in I}$ as a distinguished family of nonzero orthogonal
idempotents, let $A=B\oplus J$ be a fixed Wedderburn-Malcev type
decomposition, where $B=\oplus_{i\in I}Re_i$, and let $\mathcal{T}$
be the corresponding class of torsion $A$-modules. Suppose that, for
each $i\in I$, there is an integer $k(i)>0$ such that $J^me_i$ is a
finitely generated left ideal, for all $m\geq k(i)$. Then
$\mathcal{T}$ is a hereditary torsion class in $A-\text{Mod}$.
Moreover, if $Je_i$ is finitely presented as a left ideal, for each
$i\in I$ and $\iota :A-\text{Mod}/\mathcal{T}\longrightarrow
A-\text{Mod}$ is the section functor, then:
\begin{enumerate}
\item $\text{Im}(\iota )$ consists of the $A$-modules $Y$ such that
$\Hom_A(Be_i,Y)=0=\Ext_A^1(Be_i,Y)$, for all $i\in I$ \item
$\mathcal{T}$ is of finite type.
\end{enumerate}
\end{lema}
\begin{proof}
It is entirely similar to that of lemma \ref{lem.description of
Im(i)}. So we just mention the little changes to be made. To see
that $\mathcal{T}$ is a hereditary torsion class, only the closure
under extensions need a proof. Choosing an exact sequence
\begin{center}
$0\rightarrow T'\longrightarrow M\longrightarrow T\rightarrow 0$
\end{center}
and an element $x\in M$, as in that lemma, we get that
$J^mx\subseteq T'$ for some $m>0$. But there are $i_1,...,i_r\in I$
such that $x=\sum_{1\leq t\leq r}e_{i_t}x$. Choosing $m\geq
\text{max}\{k(i_t):$ $t=1,...,r\}$, we have that $J^me_{i_t}x$ is
finitely generated and, hence, also $J^mx$ is finitely generated.
But then there is a power $J^p$ which annihilates $J^mx$, so that
$J^{p+m}x=0$. Then $M$ is a torsion $A$-module.

As in lemma \ref{lem.description of Im(i)} one readily sees that an
$A$-module $Y$ is in $\mathcal{T}^\perp$ if, and only if,
$\Hom_A(Be_i,M)=0$ for all $i\in I$. This is because each finitely
generated module in $\mathcal{T}$ admits a finite filtration
\begin{center}
$0\subsetneq J^nx\subsetneq J^{n-1}x\subsetneq ...\subsetneq
Jx\subsetneq Ax$
\end{center}
where $J^tx/J^{t+1}x$ is a $B$-module, for all $t=0,1,...,n$.

Then, again, the proof of assertion 1 reduces to check that if
$Y\in\mathcal{T}^\perp$ and $\Ext_A^1(Be_i,Y)=0$, for all $i\in I$,
then $\Ext_A^1(T,Y)=0$, for all $T\in\mathcal{T}$. The corresponding
proof of lemma \ref{lem.description of Im(i)} can be 'copied', just
taking into account that each $B$-module $X$ fills into an exact
sequence
\begin{center}
$0\rightarrow Z\longrightarrow\oplus_{i\in
I}Be_i^{(H_i)}\longrightarrow X\rightarrow 0$,
\end{center}
for some sets $H_i$.

Finally, in order to prove assertion 2, it is enough to observe
that, due to the fact that each $Je_i$ is assumed to be a finitely
presented left ideal, we have a projective resolution of $Be_i$ as
an $A$-module
\begin{center}
$\ldots\to P^{-2}\rightarrow P^{-1}\rightarrow P^{0}\rightarrow
Be_i\rightarrow 0$,
\end{center}
where $P^{-i}$ is a finitely generated $A$-module, for $i=0,1,2$. As
a consequence, the functors $\Hom_A(Be_i,?)$ and $\Ext_A^1(Be_i,?)$
preserve direct limits and, hence, the class $\text{Im}(\iota )$ is
closed under taking direct limits in $A-\text{Mod}$.
\end{proof}
We now get a wide class of locally finitely presented categories
without nonzero flat objects.
\begin{prop} \label{prop.no flats in A-Mod/lfd}
Let $A$ be an augmented algebra with enough idempotents, with
$(e_i)_{i\in I}$ as a distinguished family of nonzero orthogonal
idempotents, let
$A=B\oplus J$ be a fixed Wedderburn-Malcev type decomposition, where
$B=\oplus_{i\in I}Re_i$, and let
$\mathcal{T}$ be the corresponding class of torsion $A$-modules.
Assume that, for each $i\in I$, $Je_i$ is finitely presented and
$J^me_i$ is finitely generated, as left ideals, for all $m>>0$. If
the following conditions hold, then $A-Mod/\mathcal{T}$ is a
locally finitely presented category without nonzero flat objects:
\begin{enumerate}
\item If $\preceq$ the smallest preorder relation in $I$
containing all the pairs $(i,j)$ such that $e_iAe_j\neq 0$, then
$(I,\preceq )$ is downward directed and does not have a minimum.
\item For each $i\in I$, the set of indices $j\prec i$
such that either $\Hom_A(Ae_i,\frac{Ae_j}{t(A)e_j}))\neq 0$ or
$$\Ext_A^1(\frac{Ae_i}{J^me_i},\frac{Ae_j}{t(A)e_j})\neq 0,$$ for some
integer $m\geq 0$, is a finite set. \item For all indices $j\preceq
i$, the left $A$-module $\frac{Ae_i}{\sum_{k\preceq j}Ae_kAe_i}$ is
annihilated by some $J^n$, with $n>0$.
\end{enumerate}
\end{prop}
\begin{proof}
By lemma \ref{lem.Mod-modulo-lfd has no flats} and proposition
\ref{prop.quotient of locally fp is locally fp}, we know that
$A-\text{Mod}/\mathcal{T}$ is locally finitely presented.

We put $\mathcal{S}_i=\{Ae_i\}$, for each $i\in I$, and will prove
that conditions 1 and 2 of proposition \ref{prop.locally coherent
without flats} hold. Indeed, if $j<i$ are any indices in $I$, then
the trace of $\bigcup_{k\leq j}\mathcal{S}_k=\{Ae_k:$ $k\leq j \}$
in $Ae_i$ is precisely $\sum_{k\leq j}Ae_kAe_i$. Then condition 2 of
proposition \ref{prop.locally coherent without flats} follows from
condition 3 in the statement.

Let us take $i\in I$ and take $X=Ae_i$ to check condition 1 of
proposition \ref{prop.locally coherent without flats}. On one side,
condition 2 in the statement gives that the set of indices $j\prec i$
such that $\Hom_A(Ae_i,\frac{Ae_j}{t(A)e_j})\neq 0$ is finite. On the other
hand suppose that $j\prec i$ and
$\Ext_A^1(T,\frac{Ae_j}{t(A)e_j})\neq 0$, for some $T\in\mathcal{T}$
which is epimorphic image of $Ae_i$. Since $T$ is cyclic we have
that $J^mT=0$, for some $m>0$, so that we have an exact sequence
\begin{center}
$0\rightarrow T'\longrightarrow\frac{Ae_i}{J^me_i}\longrightarrow
T\rightarrow 0$.
\end{center}
By applying the long exact sequence of
$\Ext(?,\frac{Ae_j}{t(A)e_j})$ and taking into account that
$\Hom_A(T',\frac{Ae_j}{t(A)e_j})=0$, we get that
$\Ext_A^1(\frac{Ae_i}{J^me_i},\frac{Ae_j}{t(A)e_j})\neq 0$. By
condition 2, we know that there are only finitely many such indices
$j$. It follows that also condition 1 in proposition
\ref{prop.locally coherent without flats} holds. Therefore
$A-\text{Mod}/\mathcal{T}$ does not have nonzero flat objects.
\end{proof}
The conditions in last proposition look quite technical. We will see
now that the representation theory of quivers with relations
provides several examples where the conditions are satisfied. Recall
that a \emph{quiver} (equivalently \emph{oriented graph}) is a
quadruple $Q=(Q_0,Q_1,o,t)$, where $Q_0$ and $Q_1$ are sets, whose
elements are calles \emph{vertices} and \emph{arrows} respectively,
and $o,t:Q_1\longrightarrow Q_0$ are maps. Given an arrow
$\alpha\in Q_1$, the vertex $o(\alpha )$ (resp $t(\alpha )$) is
called the \emph{origin} (resp. \emph{terminus}) of $\alpha$.

If $n$ is a natural number, then a \emph{path of length $n$} is just
a vertex, when $n=0$, or a concatenation
$p=\alpha_1\alpha_2...\alpha_n$ of arrows, with
$t(\alpha_i)=o(\alpha_{i+1})$ for $i=1,...,n-1$. In this later case
$o(p):=o(\alpha_1)$ and $t(p):=t(\alpha_n)$ are called the origin
and terminus of $p$, respectively. A vertex $i\in Q_0$ is considered
to be a path of length $0$ with origin and terminus equal to $i$. A
\emph{walk} or nonoriented path between two vertices $i$ and $j$ is
a sequence $i=i_0\leftrightarrow i_1\leftrightarrow
...\leftrightarrow i_r=j$, where each edge $i_k\leftrightarrow
i_{k+1}$ is eihter an arrow $i_k\rightarrow i_{k+1}$ or an arrow
$i_k\leftarrow i_{k+1}$. The quiver is \emph{connected} when,
between any two vertices of $Q$, there is a walk.

The \emph{path $($R$-)$algebra} of $Q$ is the free $R$ module $RQ$
with basis the set of paths in $Q$ on which one defines a
multiplication of paths by the rule that $p\cdot q=0$, in case
$t(p)\neq o(q)$, and $p\cdot q=pq$ is the concatenation of $p$ and
$q$, provided that $t(p)=o(q)$. The multiplication on $RQ$ extends
by $R$-linearity this multiplication of paths. If $i\in Q_0$ then
one commonly uses the symbol $e_i$ to denote the vertex $i$, when
viewed as an element of $RQ$. Note that $RQ$ is an algebra with
enough idempotents, where $(e_i)_{i\in Q_0}$ is a distinguished
family of nonzero orthogonal idempotents.

We denote by $RQ_{\geq n}$ the ideal of $RQ$ consisting of the
$R$-linear combinations of paths of length $\geq n$. An
\emph{admissible set of relations} in $RQ$ is just a subset
$\rho\subset\bigcup_{i,j\in Q_0} e_iRQe_j$ such that each $r\in\rho$
is an $R$-linear combination of paths of length $\geq 2$. A
(two-sided) ideal $I$ of $RQ$ is called \emph{admissible} when it is
generated by an admissible set of relations. An algebra $A$ with
enough idempotents is said to be \emph{given by quivers and
relations} when there is a quiver $Q$ and an admissible set of
relations $\rho$ such that $A=RQ/I$, where $I=<\rho >$ is the ideal
of $RQ$ generated by $\rho$. We still denote by $e_i$ its image by
the projection $RQ\twoheadrightarrow RQ/I=A$. Then $(e_i)_{i\in I}$
is also a distinguished family of nonzero orthogonal idempotents of
$A$ and a Wedderburn-Malcev type decomposition of $A$ is given by
$A=B\oplus J$, where $B=\oplus_{i\in Q_0}Re_i$ and $J=RQ_{\geq
1}/I$. It will be the only one that we use in the rest of the paper
and will be called the \emph{canonical Wedderburn-Malcev type
decomposition} of $A$.

We shall say that an algebra with enough idempotents is a
\emph{locally finitely presented algebra} when it is given by quiver
and relations $(Q,\rho )$, and the set $\rho e_i=\rho\cap RQe_i$ is
finite, for all $i\in Q_0$.

The quiver $Q$ is called \emph{left (resp. right) locally finite}
when each vertex is the terminus (resp. origin) of, at most, a
finite number of arrows. We say that $Q$ is locally finite when it
is left and right locally finite. The set $Q_0$ admits a canonical
preorder $\preceq$, where $i\preceq j$ if there is a path $p$ in $Q$
with $o(p)=i$ and $t(p)=j$. Note that $\preceq$ is an order relation
exactly when $Q$ does not have oriented cycles. We shall say that
$Q$ is a \emph{downward directed quiver (without minimum)} when the
preordered set $(Q_0,\preceq )$ is downward directed (without
minimum). We will say that $Q$ is a \emph{narrow quiver} if, given
vertices $j\preceq i$, the set
\begin{center}
$\{k\in Q_0:$ $k\preceq i\text{ and }k\not\preceq j\}$
\end{center}
is finite.

Recall (see \cite{R}) that if $\Delta$ is any quiver, then the
associated translation quiver $\mathbb{Z}\Delta$ has
$(\mathbb{Z}\Delta)_0=\mathbb{Z}\times\Delta_0$ as its set of
vertices and, for each arrow $\alpha :o(\alpha )\longrightarrow
t(\alpha )$ in $\Delta$, there are two arrows $(k,\alpha
):(k,o(\alpha ))\rightarrow (k,t(\alpha ))$ and
$(k,\alpha^*):(k,t(\alpha ))\longrightarrow (k+1,o(\alpha ))$ in
$\mathbb{Z}\Delta$.
\begin{ejems}
\begin{enumerate}
\item If $\Delta$ is a finite connected quiver, then
$\mathbb{Z}\Delta$ is a connected quiver which is locally finite,
downward directed and narrow.
\item If $\Delta=\mathbf{A}_\propto$: \hspace*{0.5cm} $1\rightarrow
2\rightarrow ....\rightarrow n\rightarrow
...$, then $\mathbb{Z}\Delta$ is connected, locally finite and
downward directed, but it is not narrow.
\end{enumerate}
\end{ejems}
If $A$ is given by the quiver with relations $(Q,\rho )$, then, for
every vertex $i\in Q_0$, each relation $r\in\rho e_i$ can be written
in the form $\sum_{\alpha\in Q_1, t(\alpha )=i}x(r,\alpha )\alpha$,
for uniquely determined elements $x(r,\alpha )\in RQ_{\geq
1}e_{o(\alpha )}$. Note that, since $\rho\subset\bigcup_{j,k\in
Q_0}e_jRQe_k$, there is a unique $j=j(r)\in Q_0$ such that
$e_jr\neq 0$. Then we can assume that $x(r,\alpha )=e_jx(r,\alpha
)e_{o(\alpha )}$, for each $\alpha\in Q_1$ such that $t(\alpha )=i$.
We put $o(r)=j$. The following result is well-known when $R=K$ is a
field. We give a proof in our more general setting.
\begin{lema} \label{lem.canonical projective resolution}
Let $A$ be given by a quiver with relations $(Q,\rho )$. For each
$i\in Q_0$, there is a projective resolution of $Be_i=Re_i$ in
$A-\text{Mod}$
\begin{center}
$...\rightarrow\oplus_{r\in\rho
e_i}Ae_{o(r)}\stackrel{f}{\longrightarrow}\oplus_{\alpha\in
Q_1,t(\alpha )=i}Ae_{o(\alpha
)}\stackrel{g}{\longrightarrow}Ae_i\longrightarrow Re_i\rightarrow
0$,
\end{center}
where the maps $f$ and $g$ are determined by the equalities:
\begin{center}
$f(e_{o(r)})=\sum_{\alpha\in Q_1,t(\alpha )=i}x(r,\alpha
)e_{o(\alpha )}$ \hspace*{1cm} $g(e_{o(\alpha )})=\alpha$,
\end{center}
for each $r\in\rho e_i$ and each $\alpha\in Q_1$ with $t(\alpha
)=i$.
\end{lema}
\begin{proof}
We follow standard ideas which work for unital algebras over a field
(see, e.g., \cite{BK}). The map $g':\oplus_{\alpha\in Q_1,t(\alpha
)=i}RQe_{o(\alpha )}\longrightarrow RQ_{\geq 1}$ mapping $\sum
f_\alpha\rightsquigarrow \sum f_\alpha\alpha$ is an isomorphism of
left $RQ$-modules, thus showing that the sequence
\begin{center}
$0\rightarrow \oplus_{\alpha\in Q_1,t(\alpha )=i}RQe_{o(\alpha
)}\stackrel{g'}{\longrightarrow}RQe_i\stackrel{p}{\longrightarrow}Re_i\rightarrow
0$
\end{center}
is a projective resolution of $Re_i$ as a left $RQ$-module. Applying
the functor $A\otimes_{RQ}?$ to this sequence, we obtain an exact
sequence in $A-\text{Mod}$
\begin{center}
$0\rightarrow
\text{Tor}_1^{RQ}(A,Re_i)\longrightarrow\oplus_{\alpha\in
Q_1,t(\alpha )=i}Ae_{o(\alpha
)}\stackrel{g}{\longrightarrow}Ae_i\stackrel{\bar{p}}{\longrightarrow}Re_i\rightarrow
0$.
\end{center}
The functor $A\otimes_{RQ}?$ leaves unaltered the exact sequence
$0\rightarrow\text{Ker}(\bar{p})\stackrel{incl}{\longrightarrow}
Ae_i\stackrel{\bar{p}}{\longrightarrow}Re_i\rightarrow 0$, which we
view as a sequence of left $RQ$-modules annihilated by $I=<\rho >$.
The long exact sequence of $\text{Tor}$ then gives an exact sequence
in $A-\text{Mod}$
\begin{center}
$\text{Tor}_1^{RQ}(A,Ae_i)\stackrel{u}{\longrightarrow}\text{Tor}_1^{RQ}(A,Re_i)\stackrel{0}{\longrightarrow}Ker(\bar{p})\stackrel{incl}{\longrightarrow}
Ae_i$,
\end{center}
which shows that we have an epimorphism
$\text{Tor}_1^{RQ}(A,Ae_i)\stackrel{u}{\twoheadrightarrow}\text{Tor}_1^{RQ}(A,Re_i)$.
Moreover, by a classical argument as in the case of unital algebras
(see, e.g., \cite[Chapter VIII, Section 1]{CE}), we have an
isomorphism
$\frac{Ie_i}{I^2e_i}\stackrel{\cong}{\longrightarrow}\text{Tor}_1^{RQ}(A,Ae_i)$.
We now consider the composition of morphisms in $A-\text{Mod}$
\begin{center}
$h:\frac{Ie_i}{I^2e_i}\stackrel{\cong}{\longrightarrow}\text{Tor}_1^{RQ}(A,Ae_i)\stackrel{u}{\longrightarrow}\text{Tor}_1^{RQ}(A,Re_i)\longrightarrow
\oplus_{\alpha\in Q_1,t(\alpha )=i}Ae_{o(\alpha )}=\oplus_{\alpha\in
Q_1,t(\alpha )=i}\frac{RQe_{o(\alpha )}}{Ie_{o(\alpha )}}$.
\end{center}
It is easy to check that this map is induced by the restriction to
$Ie_i$ of the canonical map $\bar{h}:RQe_i\longrightarrow
\oplus_{\alpha\in Q_1,t(\alpha )=i}RQe_{o(\alpha )}$ which takes a
path $p=q\alpha$ to $q$, where $\alpha$ is the last arrow of $p$.

If now $x\in Ie_i$ and we write it as $x=\oplus_{\alpha\in
Q_1,t(\alpha )=i}x(\alpha )\alpha$, for uniquely determined
$x(\alpha )\in RQe_{o(\alpha )}$, we have that
$h(x+I^2e_i)=\sum_{\alpha\in Q_1,t(\alpha )=i}(x(\alpha )
+Ie_{o(\alpha )})$. Then $h(x+I^2e_i)=0$ if, and only if, $x(\alpha
)\in Ie_{o(\alpha )}$ for each $\alpha\in Q_1$ with $t(\alpha )=i$.
It follows that $\text{Ker}(h)=\frac{I^2e_i+(\oplus_{\alpha\in
Q_1,t(\alpha )=i}I\alpha )}{I^2e_i}=\frac{I^2e_i+I\cdot RQ_{\geq
1}e_i}{I^2e_i}$.

The above paragraphs show that we have an exact sequence in
$A-\text{Mod}$
\begin{center}
$0\rightarrow\frac{I^2e_i+I\cdot RQ_{\geq
1}e_i}{I^2e_i}\hookrightarrow\frac{Ie_i}{I^2e_i}\stackrel{h}{\longrightarrow}\oplus_{\alpha\in
Q_1,t(\alpha )=i}Ae_{o(\alpha
)}\stackrel{g}{\longrightarrow}Ae_i\stackrel{\bar{p}}{\longrightarrow}Re_i\rightarrow
0$.
\end{center}
This shows that we have an isomorphism
$\tilde{h}:\frac{Ie_i}{I^2e_i+I\cdot RQ_{\geq
1}e_i}\stackrel{\cong}{\longrightarrow}\text{Ker}(g)$ which maps
$\sum_{\alpha\in Q_1,t(\alpha )=i}x(\alpha )\alpha +(I^2e_i+I\cdot
RQ_{\geq 1})$ onto $\sum_{\alpha\in Q_1,t(\alpha )=i}x(\alpha
)e_{o(\alpha )}$. We just need to check that
$\frac{Ie_i}{I^2e_i+I\cdot RQ_{\geq 1}}$ is generated, as a left
$A$-module, by the elements $\bar{r}=r+(I^2e_i+I\cdot RQ_{\geq 1})$,
where $r$ varies in the set $\rho e_i$. Indeed, if this is the case
then we will have a canonical epimorphism of left $A$-modules
$\oplus_{r\in\rho
e_i}Ae_{o(r)}\twoheadrightarrow\frac{Ie_i}{I^2e_i+I\cdot RQ_{\geq
1}e_i}$ mapping $e_{o(r)}\rightarrow\bar{r}$, for each $r\in\rho
e_i$. The composition of this epimorphism followed by $\tilde{h}$
and by the inclusion $\text{Ker}(g)\rightarrow\oplus_{\alpha\in
Q_1,t(\alpha )=i}Ae_{o(\alpha )}$ is just the map $f$ in the
statement and the proof will be finished.

To check what is needed, take $x\in Ie_i$. Then we can express it in
the form $x=\sum_{1\leq j\leq m}\lambda_jp_jr_jq_j$, where $0\neq
\lambda_j\in R$, $r_j\in\rho$ is a relation and $p_j$ and $q_j$ are
paths with $t(p_j)=o(r_j)$ and $o(q_j)=t(r_j)$, $t(q_j)=i$. Note
that if $\text{length}(q_j)>0$ then $p_jr_jq_j$ is in $I\cdot
RQ_{\geq 1}e_i$ and, hence, it becomes zero in
$\frac{Ie_i}{I^2e_i+I\cdot RQ_{\geq 1}e_i}$. Therefore we can assume
that $\text{length}(q_j)=0$ (i.e. $q_j=e_i$), for all $j=1,...,m$.
But then all the $r_j$ are in $\rho e_i$ and we can rewrite $x$ as
$\sum_{r\in\rho e_i}a(r)r$, for some elements $a(r)\in RQe_{o(r)}$.
Then $\frac{Ie_i}{I^2e_i+I\cdot RQ_{\geq 1}}$ is generated by the
$\bar{r}$ as a left $A$-module, as desired.
\end{proof}
\begin{prop} \label{prop.algebra with quiver-relations incomplete}
Let $Q$ be a quiver which is connected, left locally finite,
downward directed, narrow and has no oriented cycles, and let $A$
be an algebra locally finite presented by the quiver with relations
$(Q,\rho )$. Then the following assertions hold:
\begin{enumerate}
\item $Je_i$ is finitely presented and $J^me_i$ is finitely
generated, as left ideals, for all $i\in Q_0$ and all $m>0$
\item The order relation $\preceq$ is the smallest preorder relation in
$Q_0$
containing all the pairs $(i,j)$ such that $e_iAe_j\neq 0$
\item For all vertices $j\preceq i$, the left $A$-module
$\frac{Ae_i}{\sum_{k\preceq j}Ae_kAe_i}$ is annihilated by some
$J^n$, with $n>0$
\end{enumerate}
\end{prop}
\begin{proof}
A path $p$ will be called a 'nonzero path' when its image by the
projection $RQ\longrightarrow RQ/I=A$ is nonzero, where $I=<\rho >$.

1) $J^me_i$ is generated, as a left ideal,
by all nonzero paths of length $m$ with terminus $i\in Q_0$. The
left locally finite condition of $Q$ guarantees that there are only
a finite number of such paths.

With the notation of the preceding lemma, we have
$Im(g)=Ker(p)=Je_i$. The fact that $Je_i$ is finitely presented as a
left ideal follows from lemma \ref{lem.canonical projective
resolution} and the finiteness of $\rho e_i$.

2) $\preceq$ is the smallest preorder relation in $Q_0$ containing
all the pairs $(i,j)\in Q_0\times Q_0$ such that there is an arrow
$i\rightarrow j$ in $Q$. For such a pair it always happens that
$e_iAe_j\neq 0$ because a linear combination of arrows is never in
$I=<\rho >$. Conversely, if $e_iAe_j\neq 0$ then there is a nonzero
path $p$ such that $o(p)=i$ and $t(p)=j$, which implies that
$i\preceq j$.

3) Let $j\prec i$ be any vertices. By the narrow condition of $Q$,
there are $k_1,...,k_r\in Q_0$ such that
\begin{center}
$\{k\in Q_0:$ $k\preceq i\text{ and }k\not\preceq
j\}=\{k_1,...,k_r\}$.
\end{center}
Now the left locally finite condition of $Q$ and the fact that $Q$
has no oriented cycles imply that the paths with origin a $k_t$ and
terminus $i$ have a maximal length $n$. Then any path $p$ in $Q$
with $t(p)=i$ and length $m>n$ has the property that $o(p)\preceq
j$. Then we have $p\in\sum_{k\preceq j}Ae_kAe_i$, and hence
$J^me_i\subseteq \sum_{k\preceq j}Ae_kAe_i$.
\end{proof}
Recall that a hereditary torsion class $\mathcal{T}'$ in a
Grothendieck category $\mathcal{G}$ is itself a Grothendieck
category. Indeed, it is clearly abelian, with kernels and cokernels
of morphisms calculated as in $\mathcal{G}$, and it is cocomplete
with coproducts also calculated as in $\mathcal{G}$. Moreover, the
inclusion functor $j:\mathcal{T}'\hookrightarrow\mathcal{G}$
preserves colimits, since it is a left adjoint functor, and it is
exact. Then direct limits in $\mathcal{T}'$ are exact because they
are so in $\mathcal{G}$. Finally, if $G$ is a generator of
$\mathcal{G}$, then the epimorphic images of $G$ which are in
$\mathcal{T}'$ form a set of generators of $\mathcal{T}'$.

We are now ready to prove the following result, where the grading
considered in $RQ$ is the length grading, i.e., the one for which
each path is homogeneous of degree equal to its length.
\begin{teor} \label{teor.modules-modulo-locally-finite}
Let $A$ be an algebra with enough idempotents locally finitely
presented by the quiver with relations $(Q,\rho )$, where $Q$ is
connected, locally finite and has no oriented cycles. Let
$A=B\oplus J$ its canonical Wedderburn-Malcev type decomposition and
let $\mathcal{T}$ be the class of torsion $A$-modules, whose torsion
radical is denoted by $t$. The following assertions hold:
\begin{enumerate}
\item If $\rho$ consists of homogeneous relations, $t(A)=0$ and, for
each $i\in Q_0$, there are only finitely many vertices
$j$ such that $i\preceq j$, then $\mathcal{T}$ is a locally finitely
presented Grothendieck category with no nonzero flat objects.
\item If $Q$ is downward directed and narrow and
$\Ext_A^1(Be_i,\frac{Ae_j}{t(A)e_j})=0$, for all $i,j\in Q_0$, then
$\mathcal{G}/\mathcal{T}$ is a locally finitely presented
Grothendieck category with no nonzero flat objects.
\end{enumerate}
\end{teor}
\begin{proof}
1) By the proof of the preceding proposition, we know that $Je_i$ is
finitely presented and $J^me_i$ is finitely generated, as left
ideals, for all $i\in Q_0$ and all $m>0$. Then, by lemma
\ref{lem.Mod-modulo-lfd has no flats}, we know that $\mathcal{T}$ is
a hereditary torsion class of finite type in $A-\text{Mod}$.
Moreover $\mathcal{S}:=\{\frac{Ae_i}{J^me_i}:$ $i\in Q_0\text{,
}m>0\}$ is a set of finitely presented generators of $\mathcal{T}$.
On the other, hand each finitely generated $A$-module $T$ in
$\mathcal{T}$ is annihilated by some power $J^m$ and, hence, it is
canonically a $A/J^m$-module. It follows that it is a direct limit
of finitely presented $A/J^m$-modules, all of which are also
finitely presented objects of $\mathcal{T}$. The fact that each
object of $\mathcal{T}$ is a direct union of its finitely generated
submodules implies, by suitable arrangement of direct systems, that
each object of $\mathcal{T}$ is a direct limit of finitely presented
objects of $\mathcal{T}$. Therefore $\mathcal{T}$ is a locally
finitely presented Grothendieck category.

Note that, since we have epimorphisms
$\frac{Ae_i}{J^{n+1}e_i}\twoheadrightarrow\frac{Ae_i}{J^ne_i}$ for
all $i\in Q_0$ and $n>0$, it follows that the set
$\mathcal{S}(m):=\{\frac{Ae_i}{J^me_i}:$ $i\in Q_0\}$ is a set of
generators of $\mathcal{T}$, for each integer $m>0$. Suppose now
that there exists a flat object $F\neq 0$ in $\mathcal{T}$. Let us
fix any nonzero morphism $f:Ae_i\longrightarrow F$ in
$A-\text{Mod}$, which exists for some $i\in Q_0$. Then we can fix
an integer $k>0$ such that $J^ke_i\subseteq Ker(f)$ because
$\text{Im}(f)$ is cyclic and all cyclic modules in $\mathcal{T}$ are
annihilated by some power of $J$. We then consider the induced
morphism $\bar{f}:\frac{Ae_i}{J^ke_i}\longrightarrow F$. By
hypothesis, there are only finitely many vertices $j\in Q_0$ such
that $i\preceq j$. This together with the locally finite condition
of $Q$ imply that there are only finitely many paths in $Q$ with
origin $i$. Let $t$ be the maximal length of these paths. We now fix
any integer $m>k+t$ and consider an epimorphism $p:\oplus_{j\in
Q_0}\frac{Ae_j}{J^me_j}^{(H_j)}\twoheadrightarrow F$. Due to the
flatness of $F$, this epimorphism is pure and, hence, $\bar{f}$
should factor through $p$. But this is impossible since
$\Hom_A(\frac{Ae_i}{J^ke_i},\frac{Ae_j}{J^me_j})=0$, for all $j\in
Q_0$. Indeed, the existence of a nonzero morphism
$g:\frac{Ae_i}{J^ke_i}\longrightarrow\frac{Ae_j}{J^me_j}$ implies
the existence of a morphism
$\frac{Ae_i}{J^ke_i}\longrightarrow\frac{Ae_j}{J^me_j}[r]$ in
$A-Gr$, for a suitable integer $r$. Then we can assume that $g$ is a
graded morphism of some degree. It is then determined by an element
$0\neq x\in e_iAe_j$ which is an $R$-linear combination of paths
$i\rightarrow ...\rightarrow j$ of the same length, say $u$, such
that $px\in J^me_j$, for all paths $p$ of length $k$ with $t(p)=i$.
Then we have $i\preceq j$ and, hence, necessarily $u\leq t$. But if
$px\neq 0$, then $deg(px)=deg(p)+deg(x)=k+u\leq k+t<m$. This is
impossible since $px\in J^me_j$ and all nonzero homogeneous elements
of $J^me_j$ have degree $\geq m$. It follows that $px=0$ in $A$, for
all paths $p$ of length $k$ with $t(p)=i$, or, equivalently, that
$J^kx=0$. The fact that $t(A)=0$ implies that this can only happen
when $x=0$, which is a contradiction.

2) If $\Ext_A^1(Be_i,\frac{Ae_j}{t(A)e_j})=0$, for all $i,j\in Q_0$,
then, for any integer $m>0$, one easily shows by reverse induction
that $\Ext_A^1(\frac{J^ke_i}{J^me_i},\frac{Ae_j}{t(A)e_j})=0$ for
$k=m,m-1,...,0$. It follows that
$\Ext_A^1(\frac{Ae_i}{J^me_i},\frac{Ae_j}{t(A)e_j})=0$, for all
$i,j\in Q_0$. On the other hand, if
$\Hom_A(Ae_i,\frac{Ae_j}{t(A)e_j})\neq 0$ then $e_iAe_j\cong
\Hom_A(Ae_i,Ae_j)$ is nonzero, due to the projective condition of
$Ae_i$. This implies that $i\preceq j$. Then the set of indices
$j\prec i$ for which either $\Hom_A(Ae_i,\frac{Ae_j}{t(A)e_j})\neq
0$ or $\Ext_A^1(\frac{Ae_i}{J^me_i},\frac{Ae_j}{t(A)e_j})\neq 0$,
for some integer $m>0$, is the empty set. It follows that condition
2 in proposition \ref{prop.no flats in A-Mod/lfd} holds. Then, by
using proposition \ref{prop.algebra with quiver-relations
incomplete}, we see that all the hypotheses of proposition
\ref{prop.no flats in A-Mod/lfd} hold. Therefore
$A-\text{Mod}/\mathcal{T}$ is a locally finitely presented
Grothendieck category without nonzero flat objects.
\end{proof}
If $Q$ is a quiver and $K$ is a field, then the path algebra $KQ$
has a canonical structure of coalgebra, which we denote $K^\square
Q$, where the counit $\epsilon :K^\square Q\longrightarrow K$ is the
$K$-linear map which vanishes on all paths of positive length and
takes $e_i\rightsquigarrow 1$, for all $i\in Q_0$, and the
comultiplication $\Delta :K^\square Q\longrightarrow K^\square
Q\otimes K^\square Q$ maps $e_i\rightsquigarrow e_i\otimes e_i$, for
each $i\in Q_0$, and
\begin{center}
$\Delta (p)=e_{o(p)}\otimes p+\sum_{1\leq s\leq
m-1}\beta_1...\beta_s\otimes\beta_{s+1}...\beta_m+p\otimes
e_{t(p)}$,
\end{center}
for each
path of positive length $p=\beta_1\cdot ...\cdot\beta_m$.

If $K$ is a field, $(Q,\rho )$ is a quiver with relations over $K$
is a field and $I=<\rho >$ is the ideal of $KQ$ generated by $\rho$,
then
the associated ($K$-)coalgebra $C (Q,\rho )$ is the subcoalgebra
$I^\perp$ of
$K^\square Q$, where $I^\perp$ denotes the right orthogonal of $I$
with respect to the nondegenerate bilinear form $(-,-):KQ\times
K^\square Q\longrightarrow K$ which maps a pair $(p,q)$ of paths to
the Kronecker element $\delta_{pq}$. Obviously, if $\rho$ consists
of homogeneous elements, then $I$ is a graded ideal and the
coalgebra $C(Q,\rho )$ is a graded coalgebra (i.e. $\Delta
(C_n)\subseteq\sum_{r+s=n}C_r\otimes C_s$). Example 2.11 in
\cite{CS} gives a path coalgebra whose category of comodules has no
nonzero flat objects. Our next result shows that this is not an
exceptional fact for coalgebras. The reader is referred to
\cite{Sim} for the terminology and results that we use.
\begin{cor} \label{cor.comodules without flat objects}
Let $C(Q,\rho )$ be the $K$-coalgebra associated to the quiver with
relations $(Q,\rho )$. Suppose that the following conditions hold:
\begin{enumerate}
\item $Q$ is connected, locally finite and has no
oriented cycles. Moreover, for each $i\in Q_0$, there are only
finitely many $j\in Q_0$ such that $i\preceq j$ \item $\rho$
consists of homogeneous elements with respect to the length grading
on $KQ$ and, for each $i\in Q_0$, the set $\rho e_i$ is finite.
\item The algebra $A=KQ/<\rho >$ has no nonzero element $x$ such
that $A_1x=0$ (equivalently,
$C)$ contains maximal right subcomodule)
\end{enumerate}
Then the category $\text{Comod}-C$ of right $C$-comodules contains
no nonzero flat objects.
\end{cor}
\begin{proof}
Note that if $A=KQ/<\rho >$, then the category of left $A$-modules
is equivalent to the category $Rep_K(Q^{op},\rho^{op})$ of
$K$-representations of the opposite quiver with relations
$(Q^{op},\rho^{op})$. Under this equivalence, due to the properties
of our quiver, the class $\mathcal{T}$ of torsion $A$-modules
corresponds to the class $\text{Rep}_K^{lnlf}(Q^{op},\rho^{op})$ of
locally finite nilpotent representations of $(Q^{op},\rho^{op})$ in
the terminology of \cite{Sim}. Clearly $C(Q^{op},\rho^{op})$ is the
opposite coalgebra $C:=C(Q,\rho )$, i.e., if we use the classical
$\Sigma$-notation of Sweedler \cite{Swee} and $\Delta
:C\longrightarrow C\otimes C$ is the comultiplication of $C$, then
$\Delta^o:C^{op}\longrightarrow C^{op}\otimes C^{op}$ is given by
$\Delta^o(c^o)=\sum c_2^o\otimes c_1^o$, provided $\Delta (c)=\sum
c_1\otimes c_2$, and viceversa.

Then, using \cite[Theorem 4.5]{Sim}, we have equivalences of
categories
\begin{center}
$\mathcal{T}\stackrel{\cong}{\longleftrightarrow}\text{Rep}_K^{lnlf}(Q^{op},\rho^{op})\stackrel{\cong}{\longleftrightarrow}C^{op}-\text{Comod}
\stackrel{\cong}{\longleftrightarrow}\text{Comod}-C$,
\end{center}
and the result follows from theorem
\ref{teor.modules-modulo-locally-finite}.
\end{proof}
We finally give an example, derived from corollary
\ref{cor.examples}, where assertion 2 of theorem
\ref{teor.modules-modulo-locally-finite} applies. To any $n+1$
variables $x_0,...,x_n$, we associate the quiver
\medskip\par\noindent
$$\xymatrix{ \cdots -1 \ar @{->} @< 9pt> [rr]^(.6){x_{-1,0}}
\ar @{->} @<-9pt> [rr]_(.6){x_{-1,n}}^(.6){\vdots} & & 0 \ar @{->} @< 9pt>
[rr]^(.5){x_{0,0}}
\ar @{->} @<-9pt> [rr]_(.5){x_{0,n}}^(.5){\vdots} & & 1\ar @{->} @< 9pt>
[rr]^(.45){x_{1,0}}
\ar @{->} @<-9pt> [rr]_(.45){x_{1,n}}^(.45){\vdots} & & 2 \cdots } $$
\medskip\par\noindent
That is, we have $Q_0=\mathbb{Z}$ and, for each $k\in\mathbb{Z}$, we
have $n+1$ arrows $x_{k,j}:k\rightarrow k+1$, where $j=0,1,...,n$.
For each $k\in\mathbb{Z}$ and each (noncommutative) monomial $\mu
=x_{j_1}...x_{j_d}$ of degree $d$ in the $x_j$, we have a path
$\mu_k:k\rightarrow k+1\rightarrow ...\rightarrow k+d$ of length $d$
given by $\mu_k :x_{k,j_1}...x_{k+d-1,j_d}$. If now $h\in
R<x_0,x_1,...,x_n>$ is any homogeneous element of degree $d$ in the
free $R$-algebra on the $x_j$, written as $h=\sum\lambda_\mu\mu$,
where the sum ranges over the set of monomials $\mu$ of degree $d$
and $\lambda_\mu\in R$, then we put $h_k:=\sum\lambda_\mu\mu_k$.
This is a well-defined element of $RQ$.
\begin{cor} \label{cor.covering of a variety}
Let $R$ be a commutative Noetherian domain, let $\mathbf{p}$ be a prime
homogeneous ideal of $A=R[x_0,x_1,...,x_n]$
contained in $(x_0,x_1,...,x_n)^2$ and let $I_\mathbf{p}$ be the
homogeneous ideal of $RQ$ generated by all the congruencies
$x_{k,i}x_{k+1,j}-x_{k,j}x_{k+1,i}$ and all the $h_k$, where
$k\in\mathbb{Z}$ and $h$ is a homogeneous polynomial in
$\mathbf{p}$. Put $\tilde{A}=RQ/I_\mathbf{p}$ and let $\mathcal{T}$ be
the category of torsion left $\tilde{A}$-modules.

If the height of $(x_0,...,x_n)/\mathbf{p}$ in $A$ is $\geq 2$
and $A$ is either integrally closed or
Cohen-Macaulay, then $\tilde{A}-\text{Mod}/\mathcal{T}$ is a locally
finitely presented Grothendieck category with no nonzero flat
objects.
\end{cor}
\begin{proof}
By the proof of corollary \ref{cor.examples}, we know that
$\Ext_A^1(R,A)=0$. We look now at $A$ as a factor of the path
algebra of the quiver $Q'$ which has one vertex and $n+1$ loops at
it, i.e., as a factor of the free (noncommutative) $R$-algebra on
$n+1$ variables. Then $A=RQ'/I$, where $I$ is generated by a finite
set $\rho$ of homogeneous elements of $RQ'$, consisting of the
commutativity relations $x_ix_j-x_jx_i$ ($i,j=0,1,...,n$) and a
finite set of homogeneous generators of $\mathbf{p}$, which are
viewed as elements of $R<x_0,x_1,...,x_n>$. It is clear that
$\hat{\rho}=\{h_k:$ $h\in\rho\text{, }k\in\mathbb{Z}\}$ generates
$I_\mathbf{p}$ as an ideal of $RQ$. Moreover, there is a bijection
between $\rho$ and $\hat{\rho}e_k$, thus showing that $\hat{\rho}e_k$
is finite, for each $k\in\mathbb{Z}$.

We now use the theory of Galois coverings of quivers with relations
(see \cite{Gabriel2} and \cite{MVdlP}, for the old theory, and
\cite{CM} and \cite{As} for the most recent one, specially the last
paper, where the theory is developed over an arbitrary commutative
ring as we need). We know that $(Q,\hat{\rho})$ is a Galois covering
of $(Q',\rho )$. That gives a push-down functor
$F:\tilde{A}-Mod\longrightarrow A-Mod$ which takes torsion
$\tilde{A}$-modules to $A$-modules whose elements are annihilated by
powers of $(x_0,x_1,...,x_n)$. It also takes a projective resolution
of $Re_k$ to a (graded) projective resolution of $R$ as an
$A$-module. If
\begin{center}
$\rightarrow
P^{-2}\stackrel{d^{-2}}{\longrightarrow}P^{-1}\stackrel{d^{-1}}{\longrightarrow}P^0\rightarrow
Re_k\rightarrow 0$
\end{center}
is the initial part of a projective resolution as given in lemma
\ref{lem.canonical projective resolution} (with $\tilde{A}$ instead of
$A$), then we view it as a graded one, by putting $P^0=\tilde{A}e_k[0]$,
$P^{-1}=\tilde{A}e_{k-1}[-1]^{(n+1)}$ and
$P^{-2}=\oplus_{r\in\rho}\tilde{A}e_{k-d(r)}[-d(r)]$, where $d(r)$
is the degree of $r$, for each $r\in\rho$. Then the functor $F$
takes it to the following initial part of a graded projective
resolution of $R$ as an $A$-module:
\begin{center}
$\oplus_{r\in\rho}A[-d(r)]\longrightarrow A[-1]\longrightarrow
A\longrightarrow R\rightarrow 0$.
\end{center}
If $f:P^{-1}\longrightarrow \tilde{A}e_j$ is a graded morphism of
$\tilde{A}$-modules, say of degree $t$, such that $f\circ
d^{-2}=0$, then $F(f)\circ F(d^{-2})=0$ and the fact that
$\Ext_A^1(R,A)=0$, and hence $\Ext_{A-Gr}^1(R,A[t])=0$, gives a
graded morphism of $A$-modules $g:F(P^0)=R\longrightarrow
F(Ae_j)=R$ such that $g\circ F(d^{-1})=F(f)$.

By the properties of the push-down functor, we have a graded
morphism $\hat{g}:P^0=\tilde{A}e_k\longrightarrow \tilde{A}e_j$ of
the same degree that $g$ such that $F(\hat{g})=g$ and $\hat{g}\circ
d^{-1}=f$. It follows that
$\Ext_{\tilde{A}-Gr}^1(Re_k,\tilde{A}e_j[t])=0$, for all $j,k\in
Q_0$ and $t\in\mathbb{Z}$ and, hence, also
$\Ext_{\tilde{A}}^1(Be_k,\tilde{A}e_j)=\Ext_{\tilde{A}}^1(Re_k,\tilde{A}e_j)=0$,
where $B=\oplus_{i\in Q_0}Re_i$ (see the beginning of section 3).

Bearing in mind proposition \ref{prop.algebra with quiver-relations
incomplete}, we conclude that all hypotheses of assertion 2 in
theorem \ref{teor.modules-modulo-locally-finite} are satisfied by
$\tilde{A}$ and, hence, $\tilde{A}-\text{Mod}/\mathcal{T}$ is a
locally finitely presented Grothendieck category with no nonzero
flat objects.
\end{proof}


\begin{thebibliography}{99}
\bibitem{AR} {\sc J.\ Ad\'amek and J.\ Rosick\'y}, {\sl Locally
presentable and accessible
categories}. London Math. Soc Lect. Notes Ser. {\bf 189}. Cambridge
Univ. Press 1994.
\bibitem{As} {\sc H. Asashiba}, {\sl A generalization of Gabriel's
Galois covering functors and derived
equivalences}. J. Algebra {\bf 334} (2011), 109-149.
\bibitem{BK} {\sc M.C.R. Butler and A.D. King}, {\sl Minimal resolutions of
algebras.} J. Algebra {\bf 212} (1999), 323-362.
\bibitem{CE} {\sc H.\ Cartan and S.\ Eilenberg}, {\sl Homological
Algebra}, 7th edition. Princeton Univ. Press 1973.
\bibitem{CM} {\sc C. Cibils and E.N. Marcos}, {\sl Skew category, Galois
covering and smash product of a
k-category}. Proc. Amer. Math. Soc. {\bf 134}(1) (2006), 39-50.
\bibitem{CB}
{\sc W.\ Crawley-Boevey}, {\sl Locally finitely presented additive
categories.}
Comm.\ Algebra {\bf 22}(1994), 1641--1674.
\bibitem{CPT}
{\sc S.\ Crivei, M.\ Prest and B.\ Torrecillas}, {\sl Covers in finitely
accessible categories.} Proc.\ Amer.\ Math.\ Soc {\bf 138} (2010),
1213--1221.
\bibitem{CS}
{\sc J.\ Cuadra and D.\ Simson}, {\sl Flat comodules and perfect
coalgebras.}
Comm.\ Algebra {\bf 35} (2007), 3164--3194.
\bibitem{EE} {\sc E.E.\ Enochs and S.\ Estrada}, {\sl Relative homological
algebra in the category of quasi-coherent sheaves.} Adv.\ in Math.
{\bf 194} (2005), 284--295.
\bibitem{Gabriel}
{\sc P.\ Gabriel}, {\sl Des cat\'egories ab\'eliennes.} Bull.\ Soc.\ Math.\
France.
{\bf 90}(1962), 323--448.
\bibitem{Gabriel2} {\sc P. Gabriel}, {\sl The universal cover of a
representation-finite
algebra}. Springer LNM {\bf 903} (1981), 68-105.
\bibitem{G}
{\sc V.\ E.\ Govorov}, {\sl On flat modules} (in Russian). Sibirsk.\ Mat.\
\v Z. {\bf 6} (1965), 300--304.
\bibitem{GD}
{\sc A.\ Grothendieck, and J.\ A.\ Dieudonn\'e}, {\sl El\'ements de
g\'eom\'etrie alg\'ebrique I.} Grundlehren math. Wiss, {\bf 166}.
Berlin-Heidelberg-New York, Springer-Verlag, 1971.
\bibitem{H} {\sc R. Hartshorne}, {\sl Algebraic Geometry}, 6th
edition, 1993.
\bibitem{Herzog}
{\sc I.\ Herzog}, {\sl Pure-injective envelopes.} J.\ Algebra Appl. {\bf
2} (2003), 397--402.
\bibitem{Krause}
{\sc H.\ Krause}, {\sl The spectrum of a module category}, Mem. Amer. Math.
Soc., {\bf 149} (2001), no. 707.

\bibitem{Kunz}
{\sc E.\ Kunz}, {\sl Introduction to commutative algebra and algebraic
geometry}. Birkhäuser Boston, Inc., Boston, MA, 1985.

\bibitem{L}
{\sc D.\ Lazard}, {\sl Autour de la platitude.} Bull.\ Soc.\ Math.\ France
{\bf 97} (1969), 81--128.

\bibitem{MVdlP} {\sc R. Mart\'inez Villa and J. A. de la Peña}, {\sl The
universal cover of a quiver with
relations}, J.\ Pure\ Appl. Algebra {\bf 30} (1983), 277-292.

\bibitem{mat}
{\sc H.\ Matsumura}, {\sl Commutative algebra}, Second edition. Mathematics
Lecture Note Series, {\bf 56}. Benjamin/Cummings Publishing Co., Inc., Reading,
Mass., 1980.

\bibitem{Mur}{\sc D.\ Murfet},
{\sl An adjunction for modules over projective schemes}, available at
www.therisingsea.org

\bibitem{Mur2}{\sc D.\ Murfet},
{\sl Modules over projective schemes}, available at
www.therisingsea.org

\bibitem{NasFred}
{\sc C.\ N\v ast\v asescu, and F.\ Van\ Oystaeyen}, {\sl Graded and filtered
rings and modules}, Lecture Notes in Math., {\bf 758}, Springer-Verlag, Berlin,
1979.



\bibitem{Prest}
{\sc M.\ Prest, and A.\ Ralph}, {\sl Locally finitely presented categories of
sheaves of modules}, available at
http://www.maths.manchester.ac.uk/~mprest/publications.html

\bibitem{R}
{\sc C.\ Riedtmann}, {\sl Algebren, Darstellungen, \"Uberlagerungen und
zur\"uck}, Comment. Math. Helv. {\bf 55} (1980). 199--224

\bibitem{Rot}
{\sc J.J.\ Rotman}, {\sl An introduction to Homological Algebra}.
Academic Press 1979.

\bibitem{Rump2}
{\sc W.\ Rump}, {\sl Locally finitely presented categories of sheaves.} J.\
Pure\ Appl.\ Algebra {\bf 214} (2010), 177--186.

\bibitem{Rump}
{\sc W.\ Rump}, {\sl Flat covers in abelian and in non-abelian categories.}
Adv.\ Math. {\bf 225} (2010), 1589--1615.

\bibitem{Sim} {\sc D. Simson}, {\sl Path coalgebras of profinite bound quivers,
cotensor coalgebras of bound species and locally nilpotent
representations}. Colloq. Math. {\bf 109}(2) (2007), 307-343.

\bibitem{S} {\sc B.\ Stenstr\" om},
{\sl Purity in functor categories.} J.\ Algebra. {\bf 8} (1968), 352--361.

\bibitem{Sten} {\sc B.\ Stenstr\" om},
{\sl Rings of Quotients}, GMW {\bf 217}, Springer-Verlag, New York-Heidelberg
1975.

\bibitem{Swee} {\sc M.E. Sweedler}, {\sl Hopf algebras}, W. A.
Benjamin (1969).
\end{thebibliography}
\end{document}